\documentclass[12pt]{article}

\usepackage{amsfonts}
\usepackage{amsmath}

\newcommand{\beq}{\begin{equation}}
\newcommand{\eeq}{\end{equation}}
\newcommand{\bea}{\begin{eqnarray}}
\newcommand{\eea}{\end{eqnarray}}
\newcommand{\beas}{\begin{eqnarray*}}
\newcommand{\eeas}{\end{eqnarray*}}

\newtheorem{theorem}{Theorem}[section]
\newtheorem{definition}[theorem]{Definition}
\newtheorem{proposition}[theorem]{Proposition}
\newtheorem{corollary}[theorem]{Corollary}
\newtheorem{lemma}[theorem]{Lemma}
\newtheorem{remark}[theorem]{Remark}
\newtheorem{example}[theorem]{Example}
\newtheorem{examples}[theorem]{Examples}
\newtheorem{foo}[theorem]{Remarks}
\newtheorem{assumption}[theorem]{Assumption}

\newenvironment{proof}{\addvspace{\medskipamount}\par\noindent{\it
Proof}.}
{\unskip\nobreak\hfill$\Box$\par\addvspace{\medskipamount}}

\title{Convergence of optimal expected utility for a sequence of discrete-time markets in initially enlarged filtrations}
\author{Geoff Lindsell}
\date{\today}

\begin{document}
\maketitle
\begin{abstract}
In this paper, we extend Kreps' conjecture that optimal expected utility in the classic Black-Scholes-Merton (BSM) economy is the limit of optimal expected utility for a sequence of discrete-time economies in initially enlarged filtrations converge to the BSM economy in an initially enlarged filtration in a "strong" sense.  The $n$-th discrete time economy is generated by a scaled $n$-step random walk, based on an unscaled random variable $\xi$ with mean 0, variance 1 and bounded support.  Moreover, the informed insider knows each functional generating the enlarged filtrations path-by-path.  We confirm Kreps' conjecture in initially enlarged filtrations when the consumer's utility function $U$ has asymptotic elasticity strictly less than one.
\end{abstract}

\begin{section}{Introduction}
Consider a sequence of independent and identically distributed random variables, having distribution $\xi$ of mean 0, variance 1, and bounded support.  We consider economies consisting of two securities, a riskless bond with interest rate 0, and a risky security/stock, trading at times $0, \frac{1}{n},...,\frac{n-1}{n}$.  We assume the price of the stock is modeled by a geometric random walk, i.e.
\begin{align*}
S_{\frac{k}{n}} = \exp \left( \sum_{j = 1}^{k} \frac{\xi_{j}}{\sqrt{n}} \right) \text{ where } \xi_{0} = 0 \text{ and } S_{0} = 1
\end{align*}
At time 1, the bond pays a consumption dividend of 1, and the stock pays a consumption dividend of $S_{1}$ as defined above. \\
As is standard, we embed the geometric random walk into the space of continuous functions starting at 0, denoted $\Omega = C_{0}[0,1]$.  $\omega$ will denote a generic element of $\Omega$, with $\omega(t)$ denoting $\omega$ evaluated at $t$.  $\Omega$ is given the sup-norm topology with Borel $\sigma$-field $\mathcal{F}$ and standard filtration $(\mathcal{F}_{t})_{0 \leq t \leq 1}$  For each $n$, $P_{n}$ will denote the probability measure on $\Omega$ such that the joint distribution of $(\omega(0),\omega(1/n),...,\omega(1))$ is equal in distribution to that of $(\xi_{0},\xi_{1},...,\xi_{n})$, and such that $\omega(t)$ for $k/n < t < (k+1)/n$ is the linear interpolation of $\omega(k/n)$ and $\omega((k+1)/n)$.  Define $S: \Omega \rightarrow \mathbb{R}_{+}$ as $S_{t}(\omega) = e^{\omega(t)}$. \\
Recall Donsker's theorem states that $P_{n} \overset{d}{\rightarrow} P$ where $P$ is the Wiener measure on $C_{0}[0,1]$, so that under $P$, $\omega$ is a standard Brownian motion with $\omega(0) = 0$, and $S_{t}(\omega)$ under $P$ is a geometric Brownian motion, as in the traditional Black-Scholes-Merton economy. \\
As in Kreps and Schachermayer \cite{MR4154768} we initially consider a consumer wishing to maximize expected utility with an initial wealth $x$, constructing a portfolio made of the stock and the bond with "public" information flow $\mathcal{F}_{t}$.  After the initial construction of the portfolio, the investor trades in a nonanticipatory and self-financing fashion in the stock and the bond, seeking to maximize the expectation of a utility function $U : (0, \infty) \rightarrow \mathbb{R}$ at terminal time 1. \\
The original question of Kreps was if the consumer in the $n$-th discrete-time economy, where the stock and bond trade only at times $0, 1/n, 2/n,..., (n-1)/n$, does the optimal expected utility the investor can attain approach the the optimal expected utility obtained in the continuous time Black-Scholes-Merton economy. \\
A prevalent assumption in financial mathematics is a common information flow among all participants.  Recent literature, including Amendinger \cite{MR1775229, MR1954386, MR1632213} and Baudoin \cite{MR3954302, MR2021790}, has addresssed the practical importance of distinguisihing the information flow between ordinary investors and what we shall call "the informed insider."  Suppose we have two investors. The "ordinary economic agent" makes decisions according to the public information flow of the market, denoted $\mathcal{F}$, where $\mathcal{F}$ is a filtration satisfying the usual conditions; in particular, it represents the natural filtration in the BSM economy, as outlined above.  Then, the "insider" who makes decisions not only according to $\mathcal{F}$, but an additional signal modeled by the outcome of some random variable $Y$, resulting in a filtration $\mathcal{G}$ with $\mathcal{G}_{t} = \underset{\epsilon > 0}{\bigcap}\left(\mathcal{F}_{t + \epsilon} \vee \sigma(Y)\right)$.  For instance, the insider may know the terminal price of a stock, or the maximum/minimum price of a stock at a specified time horizon.  In this work, we create a new mathematical framework from the perspective of the insider, resulting in a probability space $(\Omega, \mathcal{G}, \tilde{P}_{t})$, where $\tilde{P}_{t}$ is the $[0,t]$-insider martingale measure or $[0,t]$-martingale preserving measure under initial elargement.  This name is derived from the fact that, under $\tilde{P}_{t}$, martingales with respect to the original filtration $\mathcal{F}$ remain martingales when considered as functions measurable with respect to $\mathcal{G}$.  Additionally, we construct the same framework on discrete spaces.  
Kreps' \cite{MR4154768} showed that optimal expected utility in the classic Black-Scholes-Merton economy is the limit of optimal expected utility for a sequence of discrete time economies.  In particular, letting $u_{n}^{\mathcal{F}^{n}}(x)$ denote the supremal expected utility of the ordinary investor in the $n$-th discrete-time economy, where $\mathcal{F}^{n}$ represents the natural filtration of the sequence $(\xi_{k})_{k=1}^{n}$, and letting $u^{\mathcal{F}}(x)$ denote the ordinary investor's supremal expected utility in the BSM economy, he proves $\underset{n \rightarrow \infty}{\lim}u_{n}^{\mathcal{F}^{n}}(x) = u^{\mathcal{F}}(x)$, under the assumption that the asymptotic elasticity of utility of the utility function U, as defined in Kramkov and Schachermayer \cite{MR2023886}, is less than 1.
In this work, we extend Kreps' result in the initially enlarged filtration.  In more detail, letting $u_{n}^{\mathcal{G}^{n}}(x)$ denote the supremal expected utility of the investor in the $n$-th discrete-time economy in enlarged filtrations $\mathcal{G}_{n} := \mathcal{F}^{n} \vee \sigma(Y_{n})$, and letting $u^{\mathcal{G}}(x)$ denote the investor's supremal expected utility in the BSM economy with enlarged filtration $\mathcal{G}$, we prove $\underset{n \rightarrow \infty}{\lim}u_{n}^{\mathcal{G}^{n}}(x) = u^{\mathcal{G}}(x)$, under the assumption that the asymptotic elasticity of utility of the utility function U is less than 1.
\end{section}

\begin{section}{Preliminaries}\emph{The Financial Market of the Informed Insider.} 
Throughout this work, we rely heavily on the works of Amendinger, Baudoin, and Kreps, and the framework set forth in those works.  We extend those results to include discrete economies of the form
\begin{align*}
\left(\Omega_{n}, \mathcal{H}^{n} = (\mathcal{H}_{m})_{0 \leq m \leq n}, P_{n}, S^{n} = (S_{m})_{0 \leq m \leq n} \right)
\end{align*}
where $\Omega_{n}$ represents a discrete sample space, $\mathcal{H}^{n}$ represents the discrete filtration, $P_{n}$ represents a probability measure on $\mathcal{H}^{n}$ and $S^{n}$ is the price process in this discrete economy. \\
Similarly, we consider continuous time markets of the form
\begin{align*}
(\Omega, \mathcal{H} = (\mathcal{H}_{t})_{0 \leq t \leq T}, P, S = (S_{t})_{0 \leq t \leq T})
\end{align*}
where $\Omega$ is a continuous sample space, $\mathcal{H}$ is a filtration satisfying the usual conditions, $P$ is a measure on $\mathcal{H}$, and $S$ is the price process of the market.\\
We rely heavily on the works of Kramkov and Schachermayer, especially on their notion of utility functions and results proven in \cite{MR1722287}.  In particular we have the following definition.
\begin{definition}
A \textbf{utility function} is a strictly increasing, strictly concave, and twice continuously differentiable function
\begin{align*}
U: (0, + \infty) \rightarrow \mathbb{R}
\end{align*}
which satisfies the \textbf{Inada conditions}
\begin{align*}
\lim_{x \rightarrow + \infty} U'(x) = 0, \lim_{x \rightarrow 0^{+}} U'(x) = + \infty
\end{align*}
where we use the convention that $U(x) = - \infty$ for $x \leq 0$.  We shall denote by $I$ the inverse of $U'$, and by $\tilde{U}$ the convex conjugate of $U$:
\begin{align*}
\tilde{U}(y) = \max_{x > 0}(U(x) - xy)
\end{align*}
That is, the Fenchel-Legendre transform of $-U(-x)$.
\end{definition}
$\mathcal{H}^{n}$ and $\mathcal{H}$ will at times represent the original filtration in the discrete, as well as the continuous Black-Scholes-Merton economies, denoted $\mathcal{F}^{n}$ and $\mathcal{F}$, respectively.  We next describe when $\mathcal{H}^{n}$ and $\mathcal{H}$ represent enlarged filtrations, denoted $\mathcal{G}^{n}$ and $\mathcal{G}$, enlarged by a signal/random variable $Y_{n}$ in the discrete case, enlarged by a signal/random variable $Y$ in the continuous case, observed by the "informed insider."\\
Let $\mathcal{P}_{n}$ be a discrete set containing the range of $Y_{n}$, endowed with the discrete $\sigma$-algebra and let $Y_{n}: \Omega_{n} \rightarrow \mathcal{P}_{n}$ be an $\mathcal{F}_{n}$-measurable random variable.  Denote by $P_{Y_{n}}$ the law of $Y_{n}$ and assume that $Y_{n}$ admits a regular disintegration with respect to the filtration $\mathcal{F}^{n}$.  It is common in the literature, e.g. in the works of Amendinger \cite{MR1775229, MR1954386, MR1632213} or Baudoin \cite{MR1919610, MR2021790, MR2213259} to make the following set of assumptions.
\begin{assumption}
There exists a jointly measurable and $\mathcal{F}^{n}$-adapted process 
\begin{align*}
\eta_{m}^{y,n}, \text{ } 0 \leq m < n, \text{ } y \in \mathcal{P}_{n}
\end{align*}
satisfying for all $0 \leq m < n$ and $y \in \mathcal{P}_{n}$,
\begin{align*}
P[Y_{n} = y | \mathcal{F}_{m}] = \eta_{m}^{y,n}P[Y_{n} = y]
\end{align*}
\end{assumption}
\begin{remark}
We can note here that, if we denote by $P_{n}^{y}$ the disintegrated probability measure defined by $P_{n}^{y} = P_{n}[\cdot | Y_{n} = y]$, then the above assumption implies that for $m <  n$,
\begin{align*}
P_{/\mathcal{F}_{m}}^{y} = \eta_{m}^{y,n}P_{/\mathcal{F}_{m}}
\end{align*}
In particular, for each $y \in \mathcal{P}_{n}$, the process $(\eta_{m}^{y})_{0 \leq m < n}$ is a martingale in the filtration $\mathcal{F}^{n}$.
\end{remark}
Finally, we shall denote by $\mathcal{G}^{n}$ the filtration $\mathcal{F}^{n}$ initially enlarged with $Y_{n}$, i.e. $\mathcal{G}_{m} = \mathcal{F}_{m} \vee \sigma (Y_{n})$.
Now, as in the discrete case, we will introduce the object on which anticipations will be made.  Let $\mathcal{P}$ be a Polish space endowed with its Borel $\sigma$-algebra and let $Y: \Omega \rightarrow \mathcal{P}$ be an $\mathcal{F}_{T}$-measurable random variable.  We denote by $P[Y \in y]$ the law of Y and assume that Y admits a regular disintegration with respect to the filtration $\mathcal{F}$.  That is,
\begin{assumption}
We assume there exists a jointly measurable, continuous in t, and $\mathcal{F}$-adapted process
\begin{align*}
\eta_{t}^{y}, \text{ } 0 \leq t < T, \text{ } y \in \mathcal{P},
\end{align*}
satifying for all $dt \bigotimes P_{Y}$ almost every $0 \leq t < t$ and $y \in \mathcal{P}$,
\begin{align*}
P[Y \in dy | \mathcal{F}_{t}] = \eta_{t}^{y} P[Y \in y]
\end{align*}
\end{assumption}
This is a classical assumption in the theory of the initial enlargement of the filtration $\mathcal{F}$ by Y.  This assumption is not too restrictive, and will be satisfied for "nice" functionals, to be seen.  The existence of a conditional density $\frac{P[Y \in dy | \mathcal{F}_{t}]}{P[Y \in y]}$ is the main pint, the existence of a regular version follows from general results on stochastic processes.
\begin{remark}
If we denote by $P^{y}$ the disintegrated probability measure defined by $P^{y} = P[\cdot | Y = y]$, then the above assumption implies that for $t < T$,
\begin{align*}
P_{/\mathcal{F}_{t}}^{y} = \eta_{t}^{y}P_{/\mathcal{F}_{t}}
\end{align*}
In particular, for $P_{Y}$-a.e. $y \in \mathcal{P}$, the process $(\eta_{t}^{y})_{0 \leq t < T}$ is a martingale in the filtration $\mathcal{F}$ (not uniformaly integrable).
\end{remark}
Finally, we shall denote by $\mathcal{G}$ the filtration $\mathcal{F}$ initially enlarged with Y, i.e. $\mathcal{G}_{t}$ is the $P-completion$ of $\bigcap_{\epsilon > 0} (\mathcal{F}_{t + \epsilon} \vee \sigma(Y))$, $t < T$.
\begin{subsection}{Martingale Preserving Measure in the Enlarged Filtration}
Throughout this section we assume that for $P_{n}[Y_{n} = y]$ a.e. $y \in \mathcal{P}_{n}$, the process $(\eta_{m}^{y,n})_{0 \leq m < n} := P_{n}[Y_{n} = y | \mathcal{F}_{m}]$ is strictly positive $P_{n}$ a.s. The following result is the discrete-time version of Proposition 2.3 in \cite{MR1632213} and the proof is identical. 
\begin{lemma}\cite{MR1632213}
Suppose that the regular conditional distributions of $Y_{n}$ given $\mathcal{F}_{m}$ are equivalent to the law of $Y_{n}$ for $0 \leq m < n$, i.e. for all $y \in \mathcal{P}_{n}$, the process $(\eta_{m}^{y})_{0 \leq m < n}$ is strictly positive.  Then
\begin{enumerate}
	\item $\frac{1}{\eta_{m}^{Y_{n}}}, \text{ } m < n$ is a $P_{n}$-martingale (not necessarily uniformly integrable) in the enlarged filtration $\mathcal{G}^{n}$
	\item $E_{P_{n}}[\frac{1}{\eta_{m}^{Y_{n}}} | \mathcal{F}_{m}] = 1, \text{ } m < n.$
	\item For $0 \leq m < n$, the $\sigma$-algebras $\mathcal{F}_{m}$ and $\sigma(Y_{n})$ are independent under
		\begin{align*}
			\tilde{P}^{n}_{m}(A) := \sum_{y \in A} \frac{1}{\eta_{m}^{y,n}} P_{n}[Y_{n} = y]
		\end{align*}
		i.e. , for $A_{m} \in \mathcal{F}_{m}$ and $B \in \sigma(Y_{n})$,
		\begin{align*}
			\tilde{P}^{n}_{m}[A_{m} \bigcap [Y_{n} \in B]] = P_{n}[A_{m}]P_{n}[Y_{n} \in B] = \tilde{P}^{n}_{m}[A_{m}]\tilde{P}^{n}_{m}[Y_{m} \in B]
		\end{align*}
\end{enumerate}
\end{lemma}
\begin{definition}
The probability measure $\tilde{P}^{n}_{m}$, $m < n$, definted on $\mathcal{G}_{m}$ by
\begin{align*}
\tilde{P}^{n}_{m} = \frac{1}{\eta_{m}^{Y_{n}}} P_{n /\mathcal{G}_{m}}
\end{align*}
is called the $[0,m]$-martingale preserving measure associated with $P_{n}$.
\end{definition}
\begin{remark}
As a consequence of 2. in the above lemma, we can note that for $m < n$ the law of an $\mathcal{F}^{n}$-adapted process $(X_{l})_{0 \leq l \leq m}$ is the same under $P_{n /\mathcal{F}_{m}}$ as under $\tilde{P}^{n}_{m}$.
\end{remark}
Next, suppose that the regular conditional distribution of Y given $\mathcal{F}_{t}$ are equivalent to the law of Y for all $t \in [0,T)$, that is, for all $y \in \mathcal{P}$, the process $(\eta_{t}^{y})_{0 \leq t < T}$ is strictly positive P-a.s.  The works of Amendinger and Baudoin  give us the notion of martingale preserving measure in the enlarged fitration in continuous time.  
\begin{definition}
There is a  probability measure $\tilde{P}_{t}$ for $t < T$, defined on $\mathcal{G}_{t}$ by
\begin{align*}
\tilde{P}_{t} = \frac{1}{\eta_{t}^{Y}} dP_{/\mathcal{G}_{t}}
\end{align*}
called the $[0,t]$-martingale preserving measure associated with P.
\end{definition}
\begin{remark}
The main results the reader will need to recall are Proposition 2.3, Definition 2.4, and Theorem 2.5 of \cite{MR1632213}.
\end{remark}
\end{subsection}

\begin{subsection} {Utility Maximization with Strong Information in the Binomial Model and Complete Case}
In this section we study the financial market 
\begin{align*}
(\Omega_{n}, (\mathcal{F}_{m})_{0 \leq m < n}, (S_{m})_{0 \leq m < n}, P_{n})
\end{align*}
where $(S_{m})_{0 \leq m < n}$ is a binomial random walk.  Let
\[X_{k} = \begin{cases}
-1, & \text{with probability } 1 - p \\
1, & \text{with probability } 0 < p < 1
\end{cases}\]
where $S_{n} = \sum_{k = 1}^{n} X_{k}$.  \\
Recall
\begin{theorem}
Let $\{X_{n}\}_{n=0}^{\infty}$ be a simple random walk with paramter p.  For $n \geq 1$ the distribution of the random variable $X_{n}$ is discrete with support $\{-n, -n + 2, ..., 0, ..., n-2, n\}$, and probabilities
\begin{align*}
P[X_{n} = k] = \binom{n}{\frac{n+k}{2}}p^{(n+k)/2}(1-p)^{(n-k)/2}
\end{align*}
for $k \in \{-n, -n + 2, ..., 0, ..., n-2, n\}$.
\end{theorem}
Next we compute $P[Y_{n} = y | S_{m}]$ where $Y_{n} = \underset{0 \leq m \leq n}{\max} S_{m}$.  Note that $P[Y_{n} = y | S_{m} = x] = 0$ unless $x \leq y$ and x is even.  In this case, using the fact that the increments are independent and stationary, and the fact that the random walk is a Markov Process, we have
\begin{align*}
P_{n}[Y_{n} = y | S_{m} = x] &= \frac{P_{n}[Y_{n} = y, S_{m} = x]}{P_{n}[S_{m} = x]} \\
&= \frac{P_{n}[Y_{n-m} = y - x]}{P_{n}[S_{m} = x]} \\
&= \frac{\binom{n-m}{\frac{(n-m) + (y-x) + 1}{2}}2^{-(n-m)}}{\binom{m}{\frac{m+x}{2}}p^{\frac{m+x}{2}}(1-p)^{\frac{m-x}{2}}}
\end{align*}
Thus we may take as definition
\begin{align*}
\eta_{m}^{y,n} &= P_{n}[Y_{n} = y | S_{m}] \\
&= \sum_{x \leq y} P_{n}[Y_{n} = y| S_{m} = x] P_{n}[S_{m} = x] \\
&= \sum_{x \leq y} \binom{n-m}{\lfloor \frac{(n-m) + (y - x) + 1}{2} \rfloor}2^{-(n - m)}
\end{align*}
We may then take as definition
\begin{definition}
Let $Y_{n} = \underset{0 \leq m \leq n}{\max} S_{m}$.  Then
\begin{align*}
\eta_{m}^{y,n} &= P_{n}[Y_{n} = y | S_{m}] \\
&= \sum_{x \leq y} \binom{n-m}{\lfloor \frac{(n-m) + (y - x) + 1}{2} \rfloor}2^{-(n - m)}
\end{align*}
\end{definition}
\end{subsection}
\begin{subsection}{A Convergence Result}
In order to prove Kreps' conjecture in the initially enlarged filtration, we will need to extend the discrete random variables and probability spaces to a continuous setting, as is done in Kreps' 2019, and the proof of Donsker's Theorem.  \\
For this, consider a sequence $(\xi_{j})_{j = 1}^{\infty}$ of independent, identically distributed random variables with mean zero and variance $\sigma^{2}$, $0 < \sigma^{2} < \infty$ as well as the sequence of partial sums $S_{0} = 0$, $S_{k} = \sum_{j = 1}^{k} \xi_{j}$, $k \geq 1$.  A continuous-time process $Y = (Y_{t}; t \geq 0)$ can be obtained from the sequence $(S_{k})_{k = 0}^{\infty}$ by linear interpolation; i.e.
\begin{align*}
Y_{t} = S_{\lfloor t \rfloor} + (t - \lfloor t \rfloor) \xi_{\lfloor t \rfloor + 1} \text{, } t \geq 0
\end{align*}
where $\lfloor t \rfloor$ denotes the greatest integer less than or equal to t.  Scaling appropriately both time and space, we obtain from Y a sequence of processes $(X^{n})$;
\begin{align*}
X_{t}^{(n)} = \frac{1}{\sigma \sqrt{n}} Y_{nt} \text{, } t \geq 0
\end{align*}
We will need a notation for the extension of the martingale preserving measures $\tilde{P}_{m}^{n}$ where the sample space has been enlarged to $C[0,t]$ for $t < T$ and we replace $S^{n}$ by its linear interpolation $X_{t}^{n}$.  We denote its extension by $\tilde{P}_{t}^{n}$ and continue to denote its conditional density by $\eta_{t}^{y,n}$. \\
Next, recall the following result, Theorem 4.17 of Karatzas and Shreve, p. 67
\begin{theorem} \cite{MR1121940}
With $(X^{(n)})$ defined as above and $0 \leq t_{1} \leq \cdots \leq t_{d} < \infty$, we have 
\begin{align*}
(X_{t_{1}}^{(n)}, ..., X_{t_{d}}^{(n)}) \overset{d}{\rightarrow} (W_{t_{1}},...,W_{t_{d}}) \text{ as } n \rightarrow \infty,
\end{align*}
where $(W_{t}, \mathcal{F}_{t}^{W};t \geq 0)$ is a standard, one-dimensional Brownian motion.
\end{theorem}
We will use this result in the next claim.
\begin{theorem}
Let $\{X^{(n)}\}$ be the sequence of processes defined above.  Then $P_{n}[X_{T}^{(n)} | \mathcal{F}_{t}^{n}] \rightarrow P[W_{T} | \mathcal{F}_{t}^{W}]$ P-a.s. as $n \rightarrow \infty$
\end{theorem}
\begin{proof}
By Donsker's Theorem, Theorem 3.11, and the Markov Property of a Random Walk and Brownian Motion we obtain
\begin{align*}
\eta_{t}^{y,n} &= P_{n}[X_{t}^{(n)} = y | \mathcal{F}_{nt}^{n}] \\
&= P_{n}[X_{T}^{(n)} - X_{t}^{(n)} + X_{t}^{(n)} = y | \mathcal{F}_{nt}^{n}] \\
&= P_{n}[X_{T}^{(n)} - X_{t}^{(n)} = y - x] |_{x = X_{t}^{(n)}} \\
&\overset{d}{\rightarrow} P[W_{T} - W_{t} \in dy - x]|_{x = W_{t}} \\
&= P[W_{T} - W_{t} + W_{t} \in dy | \mathcal{F}_{t}^{W}] \\
&= P[W_{T} \in dy | \mathcal{F}_{t}^{W}] \text{, P-a.s.}
\end{align*}
\end{proof}
\end{subsection}
\end{section}

\begin{section}{Convergence of Optimal Utility in the Enlarged Filtration}
We make a few necessary definitions for our results.  Let $\mathcal{H}^{n} \in \{\mathcal{F}^{n}, \mathcal{G}^{n}\}$.
\begin{definition}
A trading strategy $(\Theta_{m})_{0 \leq m \leq n} \in \mathcal{H}^{n}$ is admissible if it is self-financing and if $V_{m}(\Theta_{m}) \geq 0$ for any $m \in \{0,1,...,n\}$.  The set of admissible strategies is denoted $\mathcal{A}_{\mathcal{H}^{n}}(S^{n})$.
\end{definition}
Recall that for any equivalent martingale measure $\tilde{P}_{n}^{\mathcal{H}^{n}} \in \mathcal{M}_{\mathcal{H}^{n}}(S^{n})$, the value process $(V_{m}(\Theta_{m}))_{0 \leq m \leq n}$ is a martingale. That is, for each $0 \leq m \leq n$ we have
\begin{align*}
E_{\tilde{P}_{n}^{\mathcal{H}^{n}}}\left[V_{m+1}(\Theta_{m+1}) | \mathcal{H}_{m}\right] = V_{m}(\Theta_{m})
\end{align*}
Similarly, let $\mathcal{H} \in \{\mathcal{F}, \mathcal{G}\}$.  
\begin{definition}
$\Theta \in L^{1}(S, \mathcal{H})$ is called an admissible strategy if 
\begin{align*}
\int_{0}^{t} \Theta dS =: (\Theta \cdot S)_{t} = E_{\tilde{P}^{\mathcal{H}}}\left[ (\Theta \cdot S)_{T} | \mathcal{H}_{t} \right] \text{ for all } t \in [0, T]
\end{align*}
where $E_{\tilde{P}^{\mathcal{H}}}$ represents integration with respect to the equivalent martingale measure $\tilde{P}^{\mathcal{H}}$ under the filtration $\mathcal{H}$.  The set of admissible strategies is denoted by $\mathcal{A}_{\mathcal{H}}(S)$.
\end{definition}
\begin{definition}
Let $V_{T}$ by the terminal value of the portfolio process $V = (V_{t})_{0 \leq t \leq T}$.  For an investor with information flow $\mathcal{H}$ and initial capital $x > 0$, we define the maximum expected utility from terminal wealth by
\begin{align*}
u^{\mathcal{H}}(x) = \underset{V \in \mathcal{A}_{\mathcal{H}}(S)}{\sup}E_{P}[U(V_{T})]
\end{align*}
\end{definition}
\begin{definition}
Let $V_{n}$ by the terminal value of the portfolio process $V = (V_{m})_{0 \leq m \leq n}$.  For an investor with information flow $\mathcal{H}^{n}$ and initial capital $x > 0$, we define the maximum expected utility from terminal wealth by
\begin{align*}
u_{n}^{\mathcal{H}^{n}}(x) = \underset{V \in \mathcal{A}_{\mathcal{H}^{n}}(S^{n})}{\sup}E_{P_{n}}[U(V_{n})]
\end{align*}
\end{definition}
\begin{subsection}{Asymptotic Replication of Contingent Claims for General Random-Walk Economies}
For the sake of simplicity, we shall follow the framework of Kreps (2019) in proving our main result; that is, we shall assume that the interest rate $r = 0$, and the time horizon in both the discrete and continuous cases satisfies $T = n = 1$.  In the next result we follow 
\cite{MR3971211} closely, making the necessary adjustments for the enlarged filtration.
\begin{proposition}\cite{MR3971211}
Fix a bounded and continuous contingent claim $M$ (on $C[0,t]$, $t < 1$), measurable with respect to the filtration $\mathcal{G}$.  
\begin{enumerate}
	\item For every $\epsilon > 0$, there exists $N$ such that for all $n \geq N$, a claim $M_{n}$ can be replicated in the $n$-th discrete-			time economy with an initial investment of $E_{\tilde{P}_{t}}[M]$ such that 
		\begin{align*}
		\tilde{P}_{t}^{n}\left[|M_{n} - M| < \epsilon\right] > 1 - \epsilon
		\end{align*}
	\item Suppose that $\xi$ has bounded support.  Then $M$ can be asymptotically replicated as above where, in addition, if $V_{s}^{n}(\omega)$ is the value of the replicating portfolio in the $n$-th economy at time $s$ and in state $\omega$, then for some $\epsilon$,
	\begin{align*}
		\tilde{P}_{t}^{n}[\omega : V_{s}^{n}(\omega) \in [\underset{-}{M} - \epsilon, \overset{-}{M} + \epsilon] \text{ for all } s \in [0,t] ] = 1
	\end{align*}
\end{enumerate}
\end{proposition}
\begin{proof}
Let $Z_{t}^{\mathcal{H}}$ represent the Radon-Nikodym derivative of the equivalent martingale measure $\tilde{P}_{t}$ with respect to $P$ on $C[0,t]$, where $\mathcal{H} \in \{ \mathcal{F}, \mathcal{G}\}$ and similary for $Z_{t}^{\mathcal{H}^{n}}$ with respect to $\tilde{P}_{t}^{n}$.\\
\textbf{Step 1: } We make the following assumptions at this point
\begin{enumerate}
	\item Assume the contingent claim $M$ may be written as a bounded and continuous function of $\omega, s$.
	\item Assume that $\xi$, the unscaled "step-size" random variable, has a $N(0,1)$ distribution.  Hence, the $n$-th discrete-time economy, where trading takes place at times $s = 0, \frac{t}{n}, \frac{2t}{n},..., \frac{(n-1)t}{n}$, can be viewed in the Black-Scholes-Merton (BSM) economy where the consumer is only allowed to trade at times $0, \frac{t}{n}, $ etc.
	\item Note, that if $E_{Z_{t}^{\mathcal{G}}dP}[f(S_{t}, t) | \mathcal{G}_{u}]$ is $C^{2}$ then we may repeat the proof for any terminal noise $Y$ since $\sigma(Y) \perp \mathcal{F}$ under $\tilde{P}_{t}$.  The proof to follow specializes to the case, when the terminal noise is the final price at time 1, independent of prices at times $0 \leq t < 1$.
\end{enumerate}
By Theorem 4 of \cite{MR2021790}, for any contingent claim M, and $0 \leq s \leq t < 1$, the predictable representation property gives us
\begin{align*}
M_{s} = M_{0} + \int_{0}^{s} \Theta_{u} dS_{u}
\end{align*}
By assumption, M is a continuous function of $\omega, s$.  Write this as $M_{s} = f(S_{s}(\omega), s)$ for a continuous function f.
Since $S$ is a geometric Brownian motion, measurable with respect to the filtration $\mathcal{F}$ and $\mathcal{F} \perp \sigma(Y)$ under $\tilde{P}_{t}$, we see $E_{Z_{t}^{\mathcal{G}}dP}[f(S_{t}, t) | \mathcal{G}_{u}]$ is Markov, as $E_{Z_{t}^{\mathcal{F}}dP}[f(S_{t}, t) | \mathcal{F}_{u}]$ is Markov.  Note we have the following
\begin{align*}
Z_{t}^{\mathcal{F}} = \frac{exp(- \sigma x - \frac{1}{2} \sigma^{2} t) exp(\frac{-x^{2}}{2t})}{\sqrt{2 \pi t}}
\end{align*}
\begin{align*}
\eta_{t}^{y} = \frac{1}{\sqrt{1 - t}} exp\left( \frac{-(y - x)^{2}}{2(1 - t)} + \frac{y^{2}}{2} \right)
\end{align*}
This implies, after a bit of calculation using Theorem 3.1.2 of \cite{MR1775229}, that
\begin{align*}
Z_{t}^{\mathcal{G}} &= \frac{Z_{t}^{\mathcal{F}}}{\eta_{t}^{y}} \\
&= \sqrt{\frac{1 - t}{2\pi t}} exp \left( - \sigma x - \frac{\sigma^{2}t}{2} - \frac{x^{2}}{2t} + \frac{(y-x)^{2}}{2(1 - t)} - \frac{y^{2}}{2}\right)dxdy
\end{align*}
By integrating the density of $S_{t}$ against $Z_{t}^{\mathcal{G}}$ and using the Markov property, we obtain a version of $E_{Z_{t}^{\mathcal{G}}dP}[f(S_{t}, t) | \mathcal{G}_{u}]$, given below
\begin{align*}
f(s, u) := \sqrt{\frac{1 - (t - u)}{2\pi (t - u)}} \int_{- \infty}^{\infty} \int_{- \infty}^{\infty} exp \left( - \sigma(x - s) - \frac{\sigma^{2}(t - u)}{2} - \frac{(x - s)^{2}}{(t - u)}+ \frac{(y - (x - s))^{2}}{2(1 - (t - u))} - \frac{y^{2}}{2}\right) dxdy
\end{align*}
Note this function f is at least $C^{2}$, so we may apply the Ito Calculus.  We know from delta hedging, that
\begin{align*}
\Theta_{u} = g(S_{u}, u) = \frac{\partial f(s, u)}{\partial s} = \frac{\partial E_{\tilde{P}_{t}}[f(S_{t}, t) | S_{u} = s]}{\partial s}
\end{align*}
Thus $M_{t} = M_{0} + \int_{0}^{t} \Theta_{u} dS_{u} = f(t, 0) + \int_{0}^{t} g(S_{u}, u) dS_{u}$ where $f(t, 0) = E_{\tilde{P}_{t}}[M_{t}] = M_{0}$ and g is continuous. Next, define the forward-looking Ito sum as
\begin{align*}
I_{n}(\omega) := \sum_{k = 0}^{n - 1} g\left( S_{\frac{kt}{n}}(\omega), \frac{kt}{n} \right) \left[S_{\frac{(k+1)t}{n}}(\omega) - S_{\frac{kt}{n}}(\omega) \right]
\end{align*}
Because g is continuous, we see
\begin{align*}
\underset{n \rightarrow \infty}{\lim} I_{n} = \left[ \int_{0}^{t} g(S_{u}, u) dS_{u} \right]_{\omega} = M_{0}(\omega) - f(t,0)
\end{align*}
in probability.  Thus, for $\epsilon > 0$, we can find $N >> 0$ such that $n \geq N$ implies
\begin{align*}
\tilde{P}_{t}\left[ |I_{n} + M_{0} - M_{t}| > \epsilon \right] < \epsilon 
\end{align*}
This concludes Step 1.\\
\textbf{Step 2}:  We maintain assumption 1., but we drop assumption 2.  Recall from Theorem 3.14 that if $\{X^{(n)}\}$ is the sequence of processes defined as the interpolation of the binomial random walk, $P_{n}[X_{T}^{(n)} | \mathcal{F}_{t}^{n}] \rightarrow P[W_{T} | \mathcal{F}_{t}^{W}]$ P-a.s. as $n \rightarrow \infty$, in fact uniformly so on compact sets.  Since Donsker's Theorem tells us $P_{n} \overset{d}{\rightarrow} P$ as well, we see $\tilde{P}_{t}^{n} \overset{d}{\rightarrow} \tilde{P}_{t}$ as $n \rightarrow \infty$.  The Skorohod Represenation Theorem tells us we can find a single probability space $(\Omega, \mathbb{P})$ on which are defined random processes $\mathcal{S}$ and $\{\mathcal{S}^{n}\}$ such that $\mathcal{S}$ has the distribution of S under $\tilde{P}_{t}$, each $\mathcal{S}^{n}$ has the distribution of S under $\tilde{P}_{t}^{n}$, and the $\mathcal{S}^{n}$ converge to $\mathcal{S}$, $\mathbb{P}$-a.s. \\
Let $\epsilon > 0$.  Fix $N$ large enough so $P\left[ |I_{n} + M_{0} - M_{t}| > \epsilon \right] < \epsilon$ and $nt = N$.  Define the function $G^{N}: \mathbb{R}_{+}^{N} \rightarrow \mathbb{R}$ by
\begin{align*}
G^{N}(s_{0},...,s_{N}) := \sum_{j = 0}^{N -1} g(s_{j}, \frac{jt}{N})[s_{j+1} - s_{j}]
\end{align*}
Since g is continuous, $G^{N}$ is a continuous function of $(s_{0},...,s_{N})$.  $G^{N}(\mathcal{S}_{0}(\omega),\mathcal{S}_{\frac{t}{N}}(\omega),...,\mathcal{S}_{t}(\omega))$ is the financial-gain process for the portfolio strategy given by g applied to the price process $\mathcal{S}$ at times $0, \frac{t}{N}, \frac{2t}{N},...,\frac{(N-1)t}{N}$.  Define
\begin{align*}
\mathcal{I}(N, \omega) := G^{N}(\mathcal{S}_{0}(\omega),\mathcal{S}_{\frac{t}{N}}(\omega),...,\mathcal{S}_{t}(\omega))
\end{align*}
Similarly for nt = kN, define
\begin{align*}
\mathcal{I}^{n}(N, \omega) := G^{N}(\mathcal{S}_{0}^{n}(\omega), \mathcal{S}_{\frac{t}{N}}^{n}(\omega),...,\mathcal{S}_{t}^{n}(\omega);
\end{align*}
This is the financial-gains process for the trading strategy given by g applied to $\mathcal{S}^{n}$ at times $0, \frac{t}{N},...,\frac{(N-1)t}{N}$.  Since N is fixed, $G^{N}$ is continuous, and $\mathcal{S}^{n} \rightarrow \mathcal{S}$, we know that for $\epsilon > 0$ there exists $L >> 0$ such that for $k > L$, and $nt = kN/t$,
\begin{equation}
\mathbb{P}[|\mathcal{I}(N, \omega) - \mathcal{I}^{n}(N, \omega)| > \epsilon] < \epsilon
\end{equation}
On the probability space given by the Skorohod Representation Theorem, define $F(\omega) := f(\mathcal{S}_{t}(\omega), t)$ and $F^{n}(\omega) := f(\mathcal{S}_{t}^{n}(\omega), t)$.  The inequality
\begin{align*}
\tilde{P}_{t}\left[ |I_{n} + M_{0} - M_{t}| > \epsilon \right] < \epsilon 
\end{align*}
transferred to the Skorohod probability space is that for the value N fixed above, we have
\begin{equation}
\mathbb{P}[|\mathcal{I}(N, \omega) - F(\omega) + M_{0}| > \epsilon] < \epsilon
\end{equation}
Now, since f is continuous, the almost-sure convergence of $\{\mathcal{S}^{n}\}$ to $\mathcal{S}$ imples that for some $N_{1} >> 0$, and for all $nt > N_{1}$,
\begin{equation}
\mathbb{P}[|F^{n} - F| > \epsilon] < \epsilon
\end{equation}
Then increase $L$, above, as necessary so that if $k > L$, then $kN \geq N_{1}$.  Combining inequalities (3.1), (3.2), and (3.3) we see for the N fixed above and for all $k > L$, if $nt = kN$, then
\begin{align*}
\mathbb{P}[|M_{0} - F^{N} + \mathcal{I}^{n}(N, \omega)| > 3 \epsilon] < 3\epsilon
\end{align*}
Translating back to the original state space, this gives the result for multiples of $nt = kN$ of $N$ for which $k > L$. \\
To finish Step 2, consider contingent claims M that are bounded and continuous functions of $S_{s}$ for fixed $s < t < 1$.  Suppose $M(\omega) = f(S_{s}(\omega),s)$ for a bounded and continuous function $f(\cdot, s)$.  Note there is nothing special in the proof about the interval $[0,t]$, we could replace this with any $[0,T]$ for $T >0$. \\
Next, we extend the result from $nt = kN$ for $k > L$ (and the fixed N) to $nt > N^{*}$ for $N^{*} >> 0$.  Fix N and L from above.  For $nt > LN^{*}$, let k be such that $kN < nt \leq k(N+1)$, where $k > L$.  Suppose the claim M we are trying to nearly replicating is $M = f(S_{t}, t)$ for a continuous function $f(\cdot, t)$.  From our previous results, we can nearly replicate $f(S_{\frac{kN}{n}}, t)$ and then convert the portfolio entirely to bonds.  This will be adequate if, for any $\delta > 0$, we know that $|f(S_{t}, t) - f(S_{\frac{kN}{n}}, t)|$ is less than $\delta$ with $\tilde{P}_{t}^{n}$-probability at least $1 - \delta$. \\
But $kN < nt \leq k(N + 1)$ implies
\begin{align*}
\frac{kN}{n} < t \leq \frac{k(N+1)}{n}
\end{align*}
hence,
\begin{align*}
t - \frac{kN}{n} > 0 \geq t - \frac{kN}{n} - \frac{k}{n}
\end{align*}
and so
\begin{align*}
\frac{t}{N} > \frac{k}{n} \geq t - \frac{kN}{n} > 0
\end{align*}
Since $f(\cdot, t)$ is continuous, it is uniformly continuous on compact sets, and by tightness of the $\tilde{P}_{t}^{n}$, we know that the modulus of continuity of the paths $\omega$ can be bounded with probability approaching 1.  We can restrict attention to $\omega(s)$ lying inside some compact set of real numbers by a theorem in the Appendix, since $\omega$ that go outside a large enough compact set can be placed into the exceptional set where the estimate does not necessarily hold; puttin gthis together with the previous paragraph, using the continuity of f, and increasing N as needed, this gives us the desired result, finishing Step 2. \\
\textbf{Step 3} We now extend the result to general bounded and continuous contingent claims.  We recall that one of Prohorov's theorems tells us that that relative compactness of a sequence of probability measures is equivalent to the tightness of the same sequence.  In particular we will apply this result to the sequence $\{\tilde{P}_{t}^{n}\}$.\\
Consider claims M that are bounded and continuous functions at two (fixed) times $t_{1} < t_{2}$.  Write the claim as $f(S_{t_{1}}, S_{t_{2}})$.  For the moment, suppose that $t_{1} = \frac{k_{1}}{n}$ and $t_{2} = \frac{k_{2}}{n}$ for n large.  Suppose at time $t_{1}$, in the $n$-th economy, a consumer is provided with $E_{Z_{t}^{\mathcal{G}}dP}\left[f(S_{t_{1}}, S_{t_{2}}) | \mathcal{G}_{t_{1}} \right]$.  We assert that using the argument presented so far, we can nearly replicate $f(S_{t_{1}}, S_{t_{2}})$ by time $t_{2}$.  Note that in the original proof we started at time 0 and perhaps the $N^{*} >> 0$ chosen depends on the original starting position.  However, we can ignore any paths that take us outside a compact set and within the compact set, we can use the uniform continuity of f (and the independent increment properties of paths $\omega$) to reduce the problem to one with finitely many starting points.  And, by the sort of construction used in finishing step 2, we can "fix" things if $t_{1}$ and $t_{2}$ are not of the form $\frac{k}{n}$, as long as n is sufficiently large.\\
From here, an induction argument extends the proof to claims M that are bounded and continuous functions of S at a finite number of fixed times $t_{1} < t_{2} < ... < t_{i}$, $M = f(S_{t_{1}},S_{t_{2}},...S_{t_{i}})$.  Suppose we have the result for $i - 1$ times.  Use the proposition to prove that one can nearly replicate the claim that pays $E_{\tilde{P}_{t}}[f(S_{t_{1}},...,S_{t_{i}}) |\mathcal{G}_{t_{1}}]$, which proves a portfolio at time $t_{1}$ with enough value (and a function of $S_{t_{1}}$) to replicate (for paths that live inside a compact domain) the claim M, using the induction hypothesis. As above, the argument is complicated by the fact that the starting point is random, but the same argument provided above works. \\
Lastly, we extend the result to general bounded and continuous claims M.  Fix a bounded and continuous claim $M: C[0,t] \rightarrow \mathbb{R}$ and define claims $M^{k}$ as follows:  For any path $\omega$, let $\omega^{k}(s) = \omega(s)$ if t is of the form $\frac{jt}{k}$, and let $\omega^{k}(s)$ be the obvious linear interpolation of $\omega(\frac{jt}{k})$ and $\omega(\frac{(j+1)t}{k})$ for $\frac{jt}{k} \leq s \leq \frac{(j+1)t}{k}$.  That is,
\begin{align*}
\omega^{k}(s) = \omega\left(\frac{jt}{k}\right) + \frac{k}{t}(s - \frac{jt}{k})(\omega\left(\frac{(j+1)t}{k}\right) - \omega\left(\frac{jt}{k}\right)) \text{ for } \frac{jt}{k} \leq s \leq \frac{(j+1)t}{k}
\end{align*}
Then let $M^{k}(\omega) = M(\omega^{k})$ for each $\omega$.  Then $M^{k}$ is a bounded and continuous claim that depends only on the value of $\omega$ at times $0, \frac{t}{k}, \frac{2t}{k},...,t$.  Thus, from the previous step we can replicate each $M^{k}$ in probability. \\
We next want to show that for each $\delta > 0$, there exist $K, N$ sufficiently large so that
\begin{align*}
\tilde{P}_{t}^{n}\left[|M^{k} - M| > \delta \right] < \delta
\end{align*}
for all $k \geq K$ and $n \geq N$.  This is a consequence of the fact that $\tilde{P}_{t}^{n} \overset{d}{\rightarrow} \tilde{P}_{t}$, implying the family of distributions $\{\tilde{P}_{t}^{n}\}$ is tight.  Tightness implies there is a compact set $Z \subseteq C[0,t]$ such that $\tilde{P}_{t}^{n}[Z] \geq 1 - \frac{\delta}{2}$ for all $n$.  Of course, $\tilde{P}_{t}[Z] \geq 1 - \frac{\delta}{2}$ as well.  Since $M$ is continuous, it is uniformly continuous on Z; let $\rho$ be such that, if $||\omega - \omega'|| \leq \rho$, then $|M(\omega) - M(\omega')| \leq \delta$ for $\omega$ and $\omega'$ in $Z$.  Because $\{\tilde{P}_{t}^{n}\}$ is tight, there exists $\eta > 0$ and an integer $N$ such that, for all $n > N$, $\tilde{P}_{t}^{n}\{\omega : w_{\omega}(\eta) \geq \rho\} < \frac{\delta}{2}$, where $w_{\omega}(\eta)$ is the modulus of continuity function of $w$; that is, $w_{\omega}(\eta) = sup\{|\omega(s) - \omega(s')| \leq \eta\}$.  But now if K is such that $\frac{1}{K} < \eta$, and $\omega$ comes from the intersection of $Z$ and $\{\omega : w_{\omega}(\eta) \leq \rho\}$ (which has probability $1 - \delta$ or more, then $||\omega - \omega^{k}|| \geq \rho$ for all $kK$: On each interval $[\frac{jt}{k},\frac{(j+1)t}{k}]$, $s \in [\frac{jt}{k}, \frac{(j+1)t}{k}]$ can be no further than $\eta$ from both $\frac{jt}{k}$ and $\frac{(j+1)t}{k}$; hence, $\omega(s)$ can be no further than $\rho$ from both $\omega(\frac{jt}{k})$ and $\omega(\frac{(j+1)t}{k})$, and so it can be no further than $\rho$ from any convex combination of the two.  Hence, for $\omega$ in this intersection and for $n \geq N$, $||\omega - \omega^{n}|| \leq \rho$.  And so, for $n \geq N$, $k \geq K$, and $\omega$ in this intersection, $|M(\omega) - M_{n}(\omega)| \leq \delta$.  Then, putting all the estimates together, by choosing sufficiently small $\delta$ and $\epsilon$, we have the full result.
\end{proof}
\begin{corollary}
Suppose the contingent claim $M$ is continuous on $C[0,t]$.  Then for each $\epsilon > 0$, there exists $N$ such that, for all $n \geq N$, a claim $M_{n}$ can be replicated in the $n$-th discrete time economy such that 
\begin{align*}
\tilde{P}_{t}^{n}\left[|M_{n} - M| > \epsilon\right] < \epsilon
\end{align*}
If $M$ is integrable with respect to $Z_{t}^{\mathcal{G}}dP$, then this can be done with an initial investment of $E_{\tilde{P}_{t}}[M_{n}]$.  And if $\xi$ has bounded support, it can be done with bounded risk in the sense that, for each $\epsilon$, there exists $D_{\epsilon}$, such that if $V_{t}^{n}(\omega)$ is the value of the replicating portfolio in the $n$-th economy, then $\tilde{P}_{t}^{n}\left[|V_{t}^{n}| > D_{\epsilon}\right] = 0$ for all $n$.
\end{corollary}
\begin{proof}
Since $\tilde{P}_{t}^{n} \overset{d}{\rightarrow} \tilde{P}_{t}$, the set of measures $\{\tilde{P}_{t}^{n}\}_{n=1}^{\infty}, \tilde{P}_{t}$ is tight.  Hence, for each fixed $\epsilon$, there is a compact set $Z_{\epsilon} \subseteq C[0,t]$, such that $\tilde{P}_{t}^{n}[Z_{\epsilon}] > 1 - \frac{\epsilon}{2}$ and $\tilde{P}_{t}[Z_{\epsilon}] > 1 - \frac{\epsilon}{2}$.  Since M is continuous, it is bounded on $Z_{\epsilon}$ by a constant $B_{\epsilon}$.  Apply Proposition 3.5.2 to $M \cdot 1_{[|M| \leq B_{\epsilon}}$, but with $\frac{\epsilon}{2}$ replacing $\epsilon$ in the statement of the proposition.  The rest of the proof is clear, once we note that if $M$ is $Z_{t}^{\mathcal{G}}dP$ integrable, then $\underset{B \rightarrow \infty}{\lim}E_{Z_{t}^{\mathcal{G}}dP}[M_{B_{\epsilon}}] = E_{Z_{t}^{\mathcal{G}}dP}[M]$.
\end{proof}
\end{subsection}
\begin{subsection}{Convergence of Optimal Consumption: Discrete Time is Asymptotically at Least as Good as Continuous Time}
\begin{lemma}
Suppose that $\xi$ has bounded support.
	\begin{enumerate}
		\item (Case 1) Fix a continuous and nondecreasing utility function u whose domain is all of $\mathbb{R}$.  Fix a contingent claim $M$ that is bounded and continuous.  For every $\epsilon > 0$, there exists $N$ such that, for all $n > N$, there is a contingent claim $M_{n}$ with initial investment $M_{0}$ such that
		\begin{align*}
			|E_{\tilde{P}_{t}^{n}}[U(M_{n})] - E_{\tilde{P}_{t}}[U(M)]| < \epsilon
		\end{align*}
		\item (Case 2) Fix a utility function $U$ that is continuous and nondecreasing on $[0, \infty)$, including the case where $U(0) = - \infty$, as long as $U$ is "continuous" at 0 in the sense that $\underset{n \rightarrow \infty}{\lim} U(\frac{1}{n}) = - \infty$.  Fix a contingent claim $M$ such that $\underset{-}{M} > 0$ and $\overset{-}{M} < \infty$. (That is, $M$ is bounded away from zero below and bounded above.)  Then for each $\epsilon > 0$, there exists an N such that for all $n > N$, there is a contingent claim $M_{n}$ with initial investment $M_{0}$ such that 
		\begin{align*}
			|E_{\tilde{P}_{t}^{n}}[U(M_{n})] - E_{\tilde{P}_{t}}[U(M)]| < \epsilon
		\end{align*}
	\end{enumerate}
\end{lemma}
\begin{proof}
Let $\epsilon > 0$, and note that we can reduce $\epsilon$ as needed to obtain $\underset{-}{M}$ as needed so that $\underset{-}{M} - \epsilon > 0$. \\
Since $U$ is continuous on the compact set $[\underset{-}{M} - \epsilon, \overset{-}{M} + \epsilon]$, it is uniformly continuous and bounded on this set.  Let $\delta^{'}$ be such that, if $y, y' \in [\underset{-}{M} - \epsilon, \overset{-}{M} + \epsilon]$ and $|y - y'| < \delta^{'}$, then $|U(y) - U(y')| < \frac{\epsilon}{3}$.  Let $|U(y)| \leq B$ for all $y \in [\underset{-}{M} - \epsilon, \overset{-}{M} + \epsilon]$, and let $\delta^{''} = \epsilon /(6B)$.  Now let $\delta = \min \{\delta^{'}, \delta^{''}, \epsilon\}$. \\
The continuity of $U$ on its domain and of $M$ on $C[0,t]$ imples that $U \circ M : C[0,t] \rightarrow \mathbb{R}$ is a continuous function on $C[0,t]$ in the sup-norm topology.  It is bounded because $M$ is bounded, so $U \circ M$ is bounded.  So, since $\tilde{P}_{t}^{n} \overset{d}{\rightarrow} \tilde{P}_{t}$, we see $E_{\tilde{P}_{t}^{n}}[U(M)] \rightarrow E_{\tilde{P}_{t}}[U(M)]$.  Hence, there exists $N'$ such that for all $n > N'$, $|E_{\tilde{P}_{t}^{n}}[U(M)] - E_{\tilde{P}_{t}}[U(M)]| < \epsilon / 3$. \\
Apply Proposition 3.5.2 to $M$ and $\delta$ to see the for some $N''$ and for all $n > N''$, a claim $M_{n}$ with initial investment $M_{0}$ exists such that $M_{n}(\omega) \in [\underset{-}{M} - \epsilon, \overset{-}{M} + \epsilon]$ with $\tilde{P}_{t}^{n}$-probability 1, and $\tilde{P}_{t}^{n}[|M_{n} - M| > \delta] < \delta$.  Fix the sequence $\{M_{n}\}$. \\
Let $N = \max\{N', N''\}$.  For any $n > N$, let $\Lambda_{n} = [|M_{n} - M| < \delta]$.  Then for any $n > N$,
	\begin{align*}
	|E_{\tilde{P}_{t}}[U(M)] - E_{\tilde{P}_{t}^{n}}[U(M_{n})]| &\leq |E_{\tilde{P}_{t}}[U(M)] - E_{\tilde{P}_{t}^{n}}[U(M)]| + |E_{\tilde{P}_{t}^{n}}[U(M)] - E_{\tilde{P}_{t}^{n}}[U(M_{n})]|
	\end{align*}
	\begin{align*}
	\leq \epsilon / 3 + |E_{\tilde{P}_{t}^{n}}[U(M)] - E_{\tilde{P}_{t}^{n}}[U(M_{n})]| \\
	= \epsilon / 3 + |E_{\tilde{P}_{t}^{n}}[U(M);\Lambda_{n}] + E_{\tilde{P}_{t}^{n}}[U(M);\Lambda_{n}^{c}] - E_{\tilde{P}_{t}^{n}}[U(M_{n});\Lambda_{n}] - E_{\tilde{P}_{t}^{n}}[U(M_{n});\Lambda_{n}^{c}]| \\
	\leq \epsilon / 3 + E_{\tilde{P}_{t}^{n}}[|U(M) - U(M_{n})|;\Lambda_{n}| + E_{\tilde{P}_{t}^{n}}[|U(M) - U(M_{n})|;\Lambda_{n}^{c}]| \\
	\end{align*}
On $\Lambda_{n}$, $M_{n}$ is within $\delta$ of $M$, so $|U(M) - U(M_{n})| < \epsilon / 3$, and so $E_{\tilde{P}_{t}^{n}}[|U(M) - U(M_{n})|;\Lambda_{n}] < \epsilon / 3$.  And both $M, M_{n} \in [\underset{-}{M} - \delta], \overset{-}{M} + \delta$, $\tilde{P}_{t}^{n}$-.a.s; hence, $|U(M) - U(M_{n})| \leq 2B$ $\tilde{P}_{t}^{n}$-a.s.  Since $\tilde{P}_{t}^{n}[\Lambda_{n}^{c}] < \delta < \epsilon / (6B)$, we see $E_{\tilde{P}_{t}^{n}}[|U(M) - U(M_{n})|;\Lambda_{n}^{c}] \leq \epsilon / 3$.  Putting all these two estimates together, we have $E_{\tilde{P}_{t}}[U(M)] - E_{\tilde{P}_{t}^{n}}[U(M_{n})] < \epsilon$ for all $n > N$.
\end{proof}
\begin{proposition}
Assume that $\xi$ has bounded support and U is continuous and nondecreasing in its domain, and we are in the Black-Scholes-Merton economy, where the consumer with initial wealth $M_{0}$ and no terminal consumption.  If $M$ is a continuous contingent claim, then for each $\epsilon > 0$, there exists $N$ such that for each $n > N$, $E_{\tilde{P}_{t}}[U(M_{n})] \geq E_{\tilde{P}_{t}}[U(M)] - \epsilon$.
\end{proposition}
\begin{proof}
If $E_{\tilde{P}_{t}}[U(M)] \leq U(M_{0})$ then we are done as we may purchase $M_{0}$ bonds at time $0$ and hold until time $t$.  Thus, we assume throughout that $E_{\tilde{P}_{t}}[U(M)] > U(M_{0})$. \\
Suppose $M$ is a continuous contingent claim with $E_{Z_{t}^{\mathcal{G}}dP}[M] \leq M_{0}$, and any other constraints that are applied.  For fixed $\epsilon > 0$, construct a continuous contingent claim $M'$ that is bounded above and below in Case 1 and bounded away from 0 and bounded above in Case 2, such that $E_{Z_{t}^{\mathcal{G}}dP}[M'] \leq M_{0}$ and $E_{\tilde{P}_{t}^{n}}[U(M')]$ is within $\epsilon / 2$ of $E_{\tilde{P}_{t}}[U(M)]$, applying our previous lemma, so the result is immediate. \\
All we need to do is produce such a claim $M'$.  We handle the two cases separately. \\
In Case 1, U is defined on all of $\mathbb{R}$.  Let $M^{B}(\omega) = \min\{ M(\omega), B \}$ for large B.  Since $M^{B}(\omega) \rightarrow M(\omega)$ as $B \rightarrow \infty$, we have $E_{\tilde{P}_{t}}[U(M^{B})] \rightarrow E_{\tilde{P}_{t}}[U(M)]$ by monotone convergence.  Moreover, since $E_{\tilde{P}_{t}}[U(M)] > U(M_{0})$, we see $M(\omega) \geq M_{0} / 2$ on a set of positive measure.  So by allowing $B$ to vary, we can find a $B^{*}$ such that $E_{Z_{t}^{\mathcal{G}}dP}[M^{B^{*}}] < E_{Z_{t}^{\mathcal{G}}dP}[M] \leq M_{0}$, and such that $E_{\tilde{P}_{t}}[U(M^{B*})] \geq E_{\tilde{P}_{t}}[U(M)] - \epsilon / 2$. \\
Now bound $M^{B^{*}}$ below, by looking at $M^{B^{*}, C}(\omega) = max\{ M^{B^{*}}(\omega), -C \}$, for C large.  For any $C$, $E_{Z_{t}^{\mathcal{G}}dP}[M^{B^{*}, C}] \geq E_{Z_{t}^{\mathcal{G}}dP}[M^{B^{*}}]$, but again invoking monotone convergence, we can find C large enough so that the added cost of $E_{Z_{t}^{\mathcal{G}}dP}[M^{B^{*}, C}]$ is less than the slack in the budget constraint we obtained by replacing $M$ with $M^{B^{*}}$.  And going from $M^{B^{*}}$ to $M^{B^{*}, C}$ for any $C < \infty$ only raises the consumer's expected utility from $E_{\tilde{P}_{t}}[U(M^{B*})]$, which is already within $\epsilon / 2$ of $E_{\tilde{P}_{t}}[U(M)]$, finishing Case 1. \\
In Case 2, $U$ is defined on $[0, \infty)$.  Begin by replacing $M_{n}$ with $M^{\alpha}$ given by $M^{\alpha}(\omega) = \alpha M(\omega) + (1 - \alpha)M_{0}$, for $\alpha \in (0,1)$.  For all $\alpha$, $M^{\alpha}$ is budget feasible, since $E_{Z_{t}^{\mathcal{G}}dP}[M] \leq M_{0}$ implies that $E_{Z_{t}^{\mathcal{G}}dP}[M^{\alpha}] = \alpha E_{Z_{t}^{\mathcal{G}}dP}[M] + (1 - \alpha)M_{0} \leq M_{0}$.  And, since $M(\omega) \geq 0$, $M^{\alpha}(\omega) \geq (1 - \alpha) M_{0}$, so for $\alpha < 1$, $M^{\alpha}(\omega)$ is bounded away from zero by $(1 - \alpha)M_{0}$.  Moreover, by a double application of monotone convergence, $\underset{\alpha \rightarrow 1}{\lim}E_{\tilde{P}_{t}}[U(M^{\alpha})] = E[U(M)]$. (The integrand $M^{\alpha}$ for $\omega$ such that $M(\omega) \geq M_{0}$ is bounded below by $U(M_{0})$ and above by $U(M)$; and for $\omega$ such $M(\omega) \leq M_{0}$, it is bounded above by $U(M_{0})$ and below by $U(M)$.)  SO one can find $\alpha^{*}$ so that $E_{\tilde{P}_{t}}[U(M^{\alpha^{*}})]$ is no less than $E_{\tilde{P}_{t}}[U(M)] - \epsilon / 4$.  Now cut off $M^{\alpha^{*}}$ for very large values in the fashion of the first step in Case 1, and one produces a claim $M^{\alpha^{*}, B^{*}}$ that is budget feasible, bounded above, and bounded below away from 0, and suc that $E_{\tilde{P}_{t}}[U(M^{\alpha^{*},B^{*}})]$ is within $\epsilon / 2$ of $E_{\tilde{P}_{t}}[U(M)]$, finishing Case 2.
\end{proof}
\end{subsection}
\begin{proposition}
For the sequence of discrete-time economies described above, and a utility function $U$ satisfying the Inada conditions,
\begin{align*}
\lim \underset{n \rightarrow \infty}{\inf} u_{n}^{\mathcal{G}^{n}}(x) \geq u^{\mathcal{G}}(x) \text{ for all } x > 0
\end{align*}
\end{proposition}
\begin{proof}
Proposition 3.8 tells us that in the Black-Scholes-Merton economy, if a bounded and continuous contingent claim $M$ satisfies $E_{\tilde{P}_{t}}[U(M)] = z$ and $E_{Z_{t}^{\mathcal{G}}dP}[M] = x$ (so that $u^{\mathcal{G}}(x) \geq z$, then for every $\epsilon > 0$, there exists $N$ such that, for all $n > N$, the consumer in the $n$-th discrete-time economy can replicate a claim $M_{n}$ for an initial investment of $x$ such that $E_{\tilde{P}_{t}^{n}}[U(M_{n})] \geq z - \epsilon$. \\
Suppose we know, for instance using Proposition 4.5 of \cite{MR1954386}, that there exists a $\sigma(Y)$-measurable random variable $\Lambda_{t}(x): \Omega \rightarrow (0, + \infty)$ with 
\begin{align*}
E\left[\frac{1}{\eta_{t}^{Y}} I \left( \frac{\Lambda_{t}(x)}{\eta_{t}^{Y}} \right) | Y \right] = z,
\end{align*}
We then know, as the solution is of the form $M =  I \left( \frac{\Lambda_{t}(z)}{\eta_{t}^{Y}} \right)$ for the multiplier $\Lambda_{t}(x)$, that the solution $M : \Omega \rightarrow (0, \infty)$ is a continuous function of $\omega$.  By truncating the solution, we get approximately $u^{\mathcal{G}}(x)$, with what is a bounded and continuous claim.  Hence we conclude
\begin{align*}
\lim \underset{n \rightarrow \infty}{\inf} u_{n}^{\mathcal{G}^{n}}(x) \geq u^{\mathcal{G}}(x)
\end{align*}
Suppose now $M$ is not necessarily bounded, but $E_{\tilde{P}_{t}}[U(M)] = z$ and $E_{Z_{t}^{\mathcal{G}}dP}[X] = x$. \\
Fix $\epsilon > 0$.  First replace $M$ with a bounded claim $M'$, bounded away from $\infty$ above and away from 0 below, in two steps.  First, for $\alpha < 1$ but close to 1, let $M^{\alpha} := \alpha M + (1 - \alpha)x$.  Of course $E_{Z_{t}^{\mathcal{G}}dP}[M^{\alpha}] = x$.  And by a double application of monotone convergence (split $E_{\tilde{P}_{t}}[U(M^{\alpha})]$ into $E_{\tilde{P}_{t}}[U(M^{\alpha})1_{[M^{\alpha} \geq x]}] + E_{\tilde{P}_{t}}[U(M^{\alpha})1_{[M^{\alpha} < x]}]$), we have $\underset{\alpha \rightarrow 1}{\lim}E_{\tilde{P}_{t}}[U(M^{\alpha})] = E_{\tilde{P}_{t}}[U(M)] = z$.  So, choose $\alpha^{0}$ close enough to 1 so that $E_{\tilde{P}_{t}}[U(M^{\alpha^{0}})] \geq z - \epsilon / 4$.  Of course, $M^{\alpha^{0}}$ is bounded below by $(1 - \alpha^{0})x$.  As for the upper bound, cap $M^{\alpha^{0}}$ at some large $\beta$.  That is, let $M^{\alpha^{0}, \beta} = M^{\alpha^{0}} \wedge \beta$.  For large enough $\beta^{0}$, this is bounded above and will satisfy $E_{\tilde{P}_{t}}[U(M^{\alpha^{0}, \beta^{0}})] > z - \epsilon / 2$, while cappting $M^{\alpha^{0}}$ can only relax the budget constraint. \\
Thus, without loss of generality, we may assume that our original $M$ is bounded above and bounded away from zero.  Now apply a combination of Lusin's Theorem and Tietze's Extension Theorem to approximate M with a continuous function $M'$ that differs from $M$ on a set of arbitrarily small measure and that satisfies the same upper and lower bounds as $M$; this allows the choice of $M'$ to satisfy $E_{\tilde{P}_{t}}[U(M')] > z - 3\epsilon / 4$.  It may be that $E_{Z_{t}^{\mathcal{G}}dP}[M'] > x$, but the last $\epsilon / 4$ used to replace $M'$ with $M' - (E_{Z_{t}^{\mathcal{G}}dP}[M'] - x)$, giving a bounded and continuous contingent claim that costs x (or less) and provides expected utility $z - \epsilon$, at which point Proposition 3.8 can be applied to prove our proposition.
\end{proof}
\begin{subsection}{An Analagous Problem}
Recall from \cite{MR3971211} that we may choose one particular equivalent martingale measure, which we hereafter denote by $Z_{t}^{\mathcal{F}^{n}}$
\begin{align*}
Z_{t}^{\mathcal{F}^{n}}(\omega) = \exp[-a_{n}\omega(t) - b_{n}]
\end{align*}
for constants $a_{n}$ and $b_{n}$, chosen such that $Z_{t}^{\mathcal{F}^{n}}$ is a martingale probability measure.  In particular, $a_{n}$ are chosen to make the security a martingale under this measure, i.e.
\begin{align*}
E_{P_{n}}\left[Z_{t}^{\mathcal{F}^{n}}e^{\omega(\frac{(k+1)t}{n})} | \mathcal{F}_{\frac{kt}{n}}\right] = e^{\omega(\frac{kt}{n})}
\end{align*}
 It is shown there that $a_{n} = \frac{1}{2} +\frac{ E[\xi^{3}]}{24 \sqrt{n}} + o(\frac{1}{\sqrt{n}})$ where $E[\xi^{3}]$ is the third moment of $\xi$, and that $\underset{n \rightarrow \infty}{b_{n}} = \frac{1}{8}$.  Recall then that $P_{n} \overset{d}{\rightarrow} P$ weakly on $C_{0}[0,1]$ endowed with the sup-norm topology and, for this specific equivalent martingale measure, $Z_{t}^{\mathcal{F}^{n}}dP_{n} \overset{d}{\rightarrow} Z_{t}^{\mathcal{F}}dP$.  Recall also that $Z_{t}^{\mathcal{G}^{n}} = \frac{Z_{t}^{\mathcal{F}^{n}}}{\eta_{t}^{Y_{n}}}$ and $Z_{t}^{\mathcal{G}^{n}}dP \rightarrow Z_{t}^{\mathcal{G}}$.  With these pricing kernels now defined, we can define the following supremal utility functions
\begin{align*}
u_{n}^{Z_{t}^{\mathcal{F}^{n}}}(x) := \sup \{ E_{\tilde{P}_{t}^{n}}[U(M)] \text{, subject to } E_{Z_{t}^{\mathcal{F}^{n}}dP}[M] \leq x \}
\end{align*}
and
\begin{align*}
u_{n}^{Z_{t}^{\mathcal{G}^{n}}}(x) := \sup \{ E_{\tilde{P}_{t}^{n}}[U(M)] \text{, subject to } E_{Z_{t}^{\mathcal{G}^{n}}dP}[M] \leq x \}
\end{align*}
The point is that in the above problem the consumer faces complete markets and in the consumer's real problem, there are further constraints to replicate the contingent claim.  Thus, we know that
\begin{align*}
u_{n}^{Z_{t}^{\mathcal{F}^{n}}}(x) \geq u^{\mathcal{F}^{n}}(x) \text{ for all } x > 0 \text{ and } n \geq 0
\end{align*}
\begin{align*}
u_{n}^{Z_{t}^{\mathcal{G}^{n}}}(x) \geq u^{\mathcal{G}^{n}}(x) \text{ for all } x > 0 \text{ and } n \geq 0
\end{align*}
If we can show that $\lim \underset{n \rightarrow \infty}{\sup}u_{n}^{Z_{t}^{\mathcal{G}^{n}}}(x) = u^{\mathcal{G}}(x)$, we will establish that $\underset{n \rightarrow \infty}{\lim}u_{n}^{\mathcal{G}^{n}}(x) = u^{\mathcal{G}}(x)$.  So, this is the tactic we set forth. \\
Before beginning the proof of the main result, we make one more definition.  Recall 
\begin{align*}
Z_{t}^{\mathcal{F}}(\omega) = e^{-\frac{\omega(t)}{2} - \frac{1}{8}}
\end{align*}
which is a version of the unique continuous density process of the martingale measure with respect to the filtration $\mathcal{F}$.  Similarly,
\begin{align*}
Z_{t}^{\mathcal{G}}(\omega) = \frac{Z_{t}^{\mathcal{F}}(\omega)}{\eta_{t}^{Y}(\omega)}
\end{align*}
which is a version of the unique continuous density process of the martingale measure with respect to the filtration $\mathcal{G}$. Define 
\begin{align*}
u_{n}^{Z_{t}^{\mathcal{F}}} = \sup \{E_{P_{n}}[U(M)] : E_{P_{n}}[Z_{t}^{\mathcal{F}}M] \leq x \}
\end{align*}
and
\begin{align*}
u_{n}^{Z_{t}^{\mathcal{G}}}(x) = \sup \{E_{P_{n}}[U(M)] : E_{P_{n}}[Z_{t}^{\mathcal{G}}M] \leq x \}
\end{align*}
\end{subsection}
\begin{subsection}{If Asymptotic Elasticity is less than 1, optimal expected utilities are finite and converge}
We rely critically on the properties of utility functions and convex duals throughout this section.  In particular, we need the following proposition, which is a summary of results from \cite{MR1024460}, \cite{MR912456}, \cite{MR1722287}
\begin{proposition}
Assume that $U$ is a utility function that is strictly increasing, strictly concave, differentiable, and satisfies the Inada conditions.  it is possible that $u(x) = \infty$ (for all $x \geq 0$.  But if $u(x) < \infty$ for some $x$, hence for all $x$, it must be that $x \rightarrow u(x)$ is strictly increasing.  Moreover we have the following possibilities.
\begin{enumerate}
\item For some utility functions $U$, $\underset{x \rightarrow \infty}{\lim}u'(x) = 0$, in which case $v(y)$ is finite for all $y > 0$.  (As $\underset{y \rightarrow 0}{\lim}v(y) = \underset{x \rightarrow \infty}{\lim}u(x)$ and $\underset{y \rightarrow \infty}{\lim}v(y) = \underset{x \rightarrow 0}{\lim}u(x)$, the function $v$ can have limit $\infty$ or a finite limit as $y$ approaches 0; and $v$ can have limit $-\infty$ or a finite limit as $y \rightarrow \infty$).
\item For other utility functions $U$, $\underset{x \rightarrow \infty}{\lim}u'(x) > 0$.  If we denote $\underset{x \rightarrow \infty}{\lim}u'(x)$ by $y_{0}$, then $v(y) = \infty$ for $y < y_{0}$, while $v(y) < \infty$ for $y > y_{0}$.  As for the behavior of $v$ as $y \downarrow y_{0}$, we have the following possibilities:
	\begin{enumerate}
	\item $\underset{y \downarrow y_{0}}{\lim}v(y) = \infty$
	\item $\underset{y \downarrow y_{0}}{\lim}v(y) < \infty \text{ and } \underset{y \downarrow y_{0}}{\lim} v'(y) = \infty$
	\item $\underset{y \downarrow y_{0}}{\lim}v(y) < \infty \text{ and } \underset{y \downarrow y_{0}}{\lim} v'(y) < \infty$
	\end{enumerate}
Moreover, all are possible for any value of $y_{0}$.
\end{enumerate}
Finally, $AE(u) \leq AE(U)$; hence, $AE(U) < 1$ implies $\underset{x \rightarrow \infty}{\lim}u'(x) = 0$.  That is, asymptotic elasticity less than 1 removes the cases given by part 2.
\end{proposition}
We will need the following lemma.
\begin{lemma}
For any constant $\gamma$
\begin{align*}
\underset{n \rightarrow \infty}{\lim}E_{\tilde{P}_{t}^{n}}[e^{\gamma \omega(t)}] = E_{\tilde{P}_{t}}[e^{\gamma \omega(t)}]
\end{align*}
Consequently,
\begin{align*}
\underset{n \rightarrow \infty}{\lim}E_{P_{n}}[Z_{t}^{\mathcal{G}}] = E_{P}[Z_{t}^{\mathcal{G}}]
\end{align*}
\end{lemma}
\begin{proof}
Since $\mathcal{F} \perp \sigma(Y)$, under $\tilde{P}_{t}^{n}$, $\omega(t) = \frac{1}{\sqrt{n}} \sum_{k=1}^{\lfloor nt \rfloor} \xi_{k} + \frac{nt - \lfloor nt \rfloor}{\sqrt{n}}\xi_{\lfloor nt \rfloor + 1}$ where $\{\xi_{k}\}$ is an i.i.d. sequence of random variables with the law of $\xi$.  Since
\begin{align*}
\left| \frac{1}{\sqrt{n}}\sum_{k=1}^{\lfloor nt \rfloor} \xi_{k} - \frac{\sqrt{t}}{\sqrt{\lfloor nt \rfloor}} \sum_{k=1}^{\lfloor nt \rfloor} \xi_{k} \right| \rightarrow 0 \text{ in probability,}
\end{align*}
and, by the central limit theorem, $\frac{\sqrt{t}}{\sqrt{nt}} \sum_{k=1}^{\lfloor nt \rfloor} \xi_{k}$ converges in distribution to a normal random variable with mean zero and variance t, we see the distribution of $\omega(t)$ under $\tilde{P}_{t}^{n}$ converges to the distribution of $\omega(t)$ under $\tilde{P}_{t}$.  It follows that the moment generating functions converge as well.
\end{proof}
We now come to the main result.
\begin{theorem}
Recall that under $\tilde{P}_{t}$, $\sigma(Y) \perp \mathcal{F}$.  Also assume,
\begin{enumerate}
\item The utility function $U$ is strictly increasing, strictly concave, continuously differentiable, satisfies the Inada conditions, and that $AE(U) < 1$.
\item $E_{Z_{t}^{\mathcal{G}}dP}[M_{t} | \mathcal{G}_{u}]$ = $E_{Z_{t}^{\mathcal{G}}dP}[f(S_{t}, t) | \mathcal{G}_{u}]$ is $C^{2}$
\item There exists a constant $C > 0$ such that 
	\begin{align*}
		\frac{1}{C} \leq \frac{\eta_{t}^{y,n}}{\eta_{t}^{y}} \leq C \text{ a.s.}
	\end{align*}
\item If $T_{k} = \inf\{s \geq 0 : \eta_{s}^{Y} \geq 1/k, \eta_{s}^{Y_{n}} \geq 1/k \text{ for all n}\}$, and either 
	\begin{align*}
		\underset{n \rightarrow \infty}{\lim} E_{P_{n}}[V(yZ_{t \wedge T_{k}}^{\mathcal{G}^{n}})] \text{ exists uniformly in k }
	\end{align*}
	or
	\begin{align*}
		\underset{k \rightarrow \infty}{\lim} E_{P_{n}}[V(yZ_{t \wedge T_{k}}^{\mathcal{G}^{n}})] \text{ exists uniformly in n }
	\end{align*}
\end{enumerate}
Then, for all $x > 0$, the value function $x \rightarrow u^{\mathcal{G}}(x)$ is finite-valued and
\begin{align*}
\underset{n \rightarrow \infty}{\lim}u_{n}^{\mathcal{G}^{n}}(x) = u^{\mathcal{G}}(x)
\end{align*}
\end{theorem}
\begin{proof}
Define a sequence of stopping times $T_{k} = \inf\{s \geq 0 : \eta_{s}^{Y} \geq 1/k, \eta_{s}^{Y_{n}} \geq 1/k \text{ for all n}\}$ and note that because $\eta_{t}^{Y} > 0$ a.s., we see $\eta_{t \wedge T_{k}}^{Y} \rightarrow \eta_{t}^{Y}$ a.s. as $k \rightarrow \infty$.  We proceed as in \cite{MR4154768} in the following steps, replacing $\eta_{t}^{Y}$ with $\eta_{t \wedge T_{k}}^{Y}$, $Z_{t \wedge T_{k}}^{\mathcal{G}}$ with $Z_{t \wedge T_{k}}^{\mathcal{G}}$ in all our definitions, and letting $u_{k}^{\mathcal{G}}$, $v_{k}^{\mathcal{G}}$ denote the new primal utility functions after these replacements.  Note we are now looking at $u_{n}^{Z_{t \wedge T_{k}}^{\mathcal{G}}}(x)$ and we have the bound
\begin{align*}
\frac{1}{\eta_{t \wedge T_{k}}^{y}} \leq k \text{ a.s. }
\end{align*}
\textbf{Step 1} We claim that if $AE(U) < 1$, then $u_{k}^{\mathcal{G}}(x) < \infty$ for all $x > 0$. \\
Corollary 6.1 of \cite{MR1722287} establishes the bound
\begin{align*}
\text{There exist } L > 0 \text{ and } \alpha > 0 \text{ such that } V(y) \leq Ly^{-\alpha} \text{, for all } y \in (0, \infty)
\end{align*}
as a consequence of $AE(U) < 1$, but only for $0 < y \leq y_{0}$, for some $y_{0}$.  If we have that $V(\infty) = U(0) < 0$, we get the bound for all $y > 0$.  Without loss of generality, we may shift $U$ by a constant, so if $U(0) \geq 0$, we replace $U$ with $U(x) - U(0) - c$, for a suitable constant $c > 0$.  Then we have the above inequality for all $y \in (0, \infty)$.  \\
From this, we have the estimate
\begin{align*}
v_{k}^{\mathcal{G}}(y) &= E_{P}\left[V\left(y Z_{t \wedge T_{k}}^{\mathcal{G}}\right)\right] \\
&\leq k^{-\alpha}L E_{P}\left[y^{-\alpha}\exp\left(-\frac{\omega(t \wedge T_{k})}{2} - \frac{1}{8}\right)^{-\alpha}\right] \\
&= k^{-\alpha}L y^{-\alpha}e^{\alpha/8}E_{P}\left[exp\left(\frac{\alpha}{2}\omega(t \wedge T_{k})\right)\right] < \infty
\end{align*}
where we have used the fact that the fact that the latter expectation is just an exponential moment of a Gaussian random variable.  Thus, by Theorem 2.0 of \cite{MR1722287}, the dual value function $y \rightarrow v_{k}^{\mathcal{G}}(y)$ as well as the primal value function $x \rightarrow u_{k}^{\mathcal{G}}(x)$ have finite values. We then deduce from Proposition 3.10 that $AE(u_{k}^{\mathcal{G}}) \leq AE(U) < 1$. \\
\textbf{Step 2}  Using the notation of section 3.3, define $u_{\infty}^{Z_{t \wedge T_{k}}^{\mathcal{G}}}(x)$ for $x > 0$ by
\begin{align*}
u_{\infty}^{Z_{t \wedge T_{k}}^{\mathcal{G}}}(x) := \lim \underset{n \rightarrow \infty}{\sup} u_{n}^{Z_{t \wedge T_{k}}^{\mathcal{G}}}(x) \\
\end{align*}
Then, we claim, $u_{\infty}^{Z_{t \wedge T_{k}}^{\mathcal{G}}}(x) < \infty$ for all $x > 0$. \\
Define for each $n$ the conjugate function $v_{n}^{Z_{t \wedge T_{k}}^{\mathcal{G}}}$, which is
\begin{align*}
v_{n}^{Z_{t \wedge T_{k}}^{\mathcal{G}}}(y) = E_{P_{n}}\left[V(yZ_{t \wedge T_{k}}^{\mathcal{G}})\right]
\end{align*}
By the same argument as in Step 1, we have
\begin{align*}
v_{n}^{Z_{t \wedge T_{k}}^{\mathcal{G}}}(y) &= E_{P_{n}}[V(yZ_{t \wedge T_{k}}^{\mathcal{G}})] \\
& \leq k^{-\alpha}L y^{-\alpha}e^{\alpha / 8} E_{P}\left[\exp\left(\frac{\alpha}{2} \omega(t \wedge T_{k}) \right) \right] \\
& < \infty
\end{align*}
By Lemma 3.11.  This implies that $v_{n}^{Z_{t \wedge T_{k}}^{\mathcal{G}}}(y)$ is uniformly bounded in n for fixed y, which, by standard arguments concerning conjugate functions, proves that $u_{\infty}^{Z_{t \wedge T_{k}}^{\mathcal{G}}}(x) < \infty$ for each $x$. \\
\textbf{Step 3} In fact, $\underset{n \rightarrow \infty}{\lim}u_{n}^{Z_{t \wedge T_{k}}^{\mathcal{G}}}(x)$ exists and equals $u_{k}^{\mathcal{G}}(x)$ for all $x > 0$. \\
Using stadard arguments concerning conjugate functions, we only need to prove $\underset{n \rightarrow \infty}{\lim}v_{n}^{Z_{t \wedge T_{k}}^{\mathcal{G}}}(y) = v_{k}^{\mathcal{G}}(y)$ here. Recall 
\begin{align*}
v_{n}^{Z_{t \wedge T_{k}}^{\mathcal{G}}}(y) = E_{P_{n}}[V(yZ_{t \wedge T_{k}}^{\mathcal{G}})] \text{ and } v_{k}^{\mathcal{G}}(y) = E_{P}[V(yZ_{t \wedge T_{k}}^{\mathcal{G}})]
\end{align*}
If V were bounded (as V is already continuous), then the conclusion would follow from the fact that $P_{n} \overset{d}{\rightarrow} P$.  Since V is typically not bounded, we must show that the contributions from the tails can be uniformly controlled in some sense.  The strategy is to show the two uniform bounds below.
\begin{align*}
\text{For each } y > 0 \text{ and } \epsilon > 0 \text{, there exists } M > 0 \text{ such that } \\
E_{P_{n}}\left[\left|V(yZ_{t \wedge T_{k}}^{\mathcal{G}})\right|1_{[V(yZ_{t \wedge T_{k}}^{\mathcal{G}})] < -M}\right] < \epsilon \text{, uniformly in n} \\
\text{For each } y > 0 \text{ and } \epsilon > 0 \text{, there exists } M > 0 \text{ such that } \\
E_{P_{n}}\left[\left|V(yZ_{t \wedge T_{k}}^{\mathcal{G}})\right|1_{[V(yZ_{t \wedge T_{k}}^{\mathcal{G}})] > M}\right] < \epsilon \text{, uniformly in n}
\end{align*}
We begin with the first inequality.  If $U(0)$ is finite, then $V(y) \geq V(\infty) = U(0)$ for all $y$, so taking $M = -U(0)$ we obtain the result immediately.  Now suppose $U(0) = -\infty$.  In this case, the Inada conditions imply $\underset{y \rightarrow \infty}{\lim}V'(y) = - \underset{y \rightarrow \infty}{\lim}(U')^{-1}(y) = 0$.  As V is convex, we see that for $M$ large, $|V(y)| \leq \epsilon y$, provided $V(y) \leq -M$.  But then 
\begin{align*}
E_{P_{n}}\left[\left|V(yZ_{t \wedge T_{k}}^{\mathcal{G}})\right|1_{[V(yZ_{t \wedge T_{k}}^{\mathcal{G}}) \leq -M}\right] & \leq E_{P_{n}}[\epsilon Z_{t \wedge T_{k}}^{\mathcal{G}}] \\
&\leq \epsilon E_{P_{n}}[Z_{t \wedge T_{k}}^{\mathcal{G}}] \\
&\leq \epsilon E_{\tilde{P}_{t}^{n}}[Z_{t}^{\mathcal{F}}] \\
& \rightarrow \epsilon \cdot 1 = \epsilon
\end{align*}
where we used Lemma 3.11 in the last step.  This shows the first inequality.  We now show the second inequality.  For the bounds $L > 0$ and $\alpha > 0$ such that $V(y) \leq Ly^{-\alpha}$ for all $y \in (0, \infty)$, choose B large enough so
\begin{align*}
E_{\tilde{P}_{t}^{n}}\left[e^{\alpha \omega(t)/2} 1_{[\omega(t) \geq B]} \right] \leq \frac{\epsilon}{k^{-\alpha}Ly^{-\alpha}e^{\alpha/8}} \text{ for all n.}
\end{align*}
The existence of such a B follows from Lemma 3.11 and, because $\tilde{P}_{t}^{n} \overset{d}{\rightarrow} \tilde{P}_{t}$,
\begin{enumerate}
\item $E_{\tilde{P}_{t}^{n}}[\min\{e^{\alpha \omega(t)}, B \}] \rightarrow E_{\tilde{P}_{t}}[\min \{e^{\alpha \omega(t)}, B\}]$ and
\item $\tilde{P}_{t}^{n}[\omega(t) \geq B] \rightarrow \tilde{P}_{t}[\omega(t) \geq B]$, for all $B > 0$
\end{enumerate}
And let
\begin{align*}
M = k^{-\alpha}Ly^{-\alpha}\exp \left( \frac{\alpha B}{2} + \frac{\alpha}{8}\right)
\end{align*}
We have $V(y) \leq Ly^{-\alpha}$ for all $y > 0$, and so
\begin{align*}
[V(yZ_{t \wedge T_{k}}^{\mathcal{G}}) \geq M] &\subseteq [L(yZ_{t \wedge T_{k}}^{\mathcal{G}})^{-\alpha} \geq M] \subseteq [k^{-\alpha}L(yZ_{t}^{\mathcal{F}})^{-\alpha} \geq k^{-\alpha}Ly^{-\alpha}e^{\alpha B/2 + \alpha /8}] \\
&= [(Z_{t}^{\mathcal{F}})^{-\alpha} \geq e^{\alpha B/2 + \alpha / 8}] =[(e^{-\omega(t)/2 - 1/8})^{-\alpha} \geq e^{\alpha B/2 + \alpha / 8}] \\
&= [e^{\omega(t)/2} \geq e^{\alpha B/2}] = [\omega(t) \geq B]
\end{align*}
Hence,
\begin{align*}
E_{P_{n}}[V(yZ_{t \wedge T_{k}}^{\mathcal{G}}) 1_{[V(yZ_{t \wedge T_{k}}^{\mathcal{G}}) \geq M]}] &\leq E_{P_{n}}[L(yZ_{t \wedge T_{k}}^{\mathcal{G}})^{-\alpha} 1_{[\omega(t) \geq B]}] \\
&= E_{P_{n}}[k^{-\alpha}Ly^{-\alpha}(e^{-\omega(t)/2 - 1/8})^{-\alpha} 1_{[\omega(t) \geq B]}] \\
&= k^{-\alpha}Ly^{-\alpha}e^{\alpha / 8} E_{P_{n}}[e^{\omega(t)/2} 1_{[\omega(t) \geq B]}] \\
&\leq \epsilon
\end{align*}
Having shown the two desired inequalities, we finish with a standard argument.  Fix $y, \epsilon > 0$, and find $M$ such that the two desired inequalities hold.  Let $V^{M}(yZ_{t \wedge T_{k}}^{\mathcal{G}}) := \max\{-M, \min\{M, V(yZ_{t \wedge T_{k}}^{\mathcal{G}})\}\}$; that is, $V^{M}(yZ_{t \wedge T_{k}}^{\mathcal{G}})$ is $V(yZ_{t \wedge T_{k}}^{\mathcal{G}})$ truncated on $[-M, M]$.  The truncation is bounded and continuous, so $P_{n} \overset{d}{\rightarrow} P$ implies $E_{P_{n}}[V^{M}] \rightarrow E_{P}[V^{M}]$.  And the difference $E_{P_{n}}[V(yZ_{t \wedge T_{k}}^{\mathcal{G}}) - V^{M}(yZ_{t \wedge T_{k}}^{\mathcal{G}})]$ are uniformly bounded by $2 \epsilon$.  Hence, $v_{k, n}^{\mathcal{G}}(y) = E_{P_{n}}[V(yZ_{t \wedge T_{k}}^{\mathcal{G}})] \rightarrow E_{P}[V(yZ_{t \wedge T_{k}}^{\mathcal{G}})] = v_{k}^{\mathcal{G}}(y)$ for all y, which implies that $u_{n}^{Z_{t \wedge T_{k}}^{\mathcal{G}}}(x) \rightarrow u_{k}^{\mathcal{G}}(x)$ for all $x > 0$. \\
\textbf{Step 4} For every $y > 0$ and $\epsilon > 0$, there exists $M > 0$ such that
\begin{align*}
E_{P_{n}}\left[\left|V(yZ_{t \wedge T_{k}}^{\mathcal{G}^{n}})\right|1_{[V(yZ_{t \wedge T_{k}}^{\mathcal{G}^{n}}) < -M]} \right] < \epsilon \text{, uniformly in n.}
\end{align*}
We may use the same proof as in Step 3.  Indeed, we want to uniformly control the right-hand tail of $v_{n}^{Z_{t \wedge T_{k}}^{\mathcal{G}^{n}}}(y) = E_{P_{n}}[V(yZ_{t \wedge T_{k}}^{\mathcal{G}^{n}})]$.  The key is that the inequality
\begin{align*}
E_{\tilde{P}_{t}^{n}}\left[e^{\alpha \omega(t)/2} 1_{[\omega(t) \geq B]} \right] \leq \frac{\epsilon}{k^{-\alpha}Ly^{-\alpha}e^{\alpha/8}} \text{ for all n.}
\end{align*}
still holds by the same argument, and we even have $E_{P_{n}}[Z_{t}^{\mathcal{G}^{n}}] = 1$ (as opposed to $E_{P_{n}}[Z_{t}^{\mathcal{G}}] \rightarrow 1$ as before). \\
\textbf{Step 5}  There is a constant $C > 1$ such that, for all $n$,
\begin{align*}
\frac{1}{C} \leq \frac{Z_{t}^{\mathcal{G}^{n}}(\omega)}{Z_{t}^{\mathcal{G}}(\omega)} \leq C \text{, } \tilde{P}_{t}^{n}-a.s.
\end{align*}
We have
\begin{align*}
\frac{Z_{t}^{\mathcal{F}^{n}}(\omega)}{Z_{t}^{\mathcal{F}}(\omega)} = \frac{e^{-a_{n}\omega(t) - b_{n}}}{e^{-\omega(t)/2 - 1/8}},
\end{align*}
where $a_{n} = 1/2 + d/\sqrt{n} + o(1/\sqrt{n})$, for $d = E[\xi^{3}]/24$, and $b_{n} = 1/8 + o(1)$.  Then,
\begin{align*}
\frac{Z_{t}^{\mathcal{F}^{n}}(\omega)}{Z_{t}^{\mathcal{F}}(\omega)} = \exp \left[\frac{d\omega(t)}{\sqrt{n}} + \omega(t) \cdot o\left(\frac{1}{\sqrt{n}}\right)+ \omega(t)\right]
\end{align*}
Since we assume $\xi$ has bounded support, there exists a constant K such that $|\xi| \leq K$ with probability 1, so $|\omega(t)| \leq K\sqrt{n}$, $\tilde{P}_{t}^{n}$-a.s..  Combining this with the assumption that there exists a constant $C > 0$ such that
\begin{align*}
\frac{1}{C} \leq \frac{\eta_{t}^{y,n}}{\eta_{t}^{y}} \leq C \text{ a.s.}
\end{align*}
By manipulating the constant $C$, it is now evident that these inequalities allow us to choose an appropriate $C > 1$ such that the desired inequality holds. \\
\textbf{Step 6} For every $y > 0$ and $\epsilon > 0$, there exists $M > 0$ such that
\begin{align*}
E_{P_{n}}[V(yZ_{t \wedge T_{k}}^{\mathcal{G}^{n}})1_{[V(yZ_{t \wedge T_{k}}^{\mathcal{G}^{n}}) > M]}] < \epsilon \text{, uniformly in n}
\end{align*}
Using the inequality from Step 5, we see $Z_{t \wedge T_{k}}^{\mathcal{G}}(\omega)/C \leq Z_{t \wedge T_{k}}^{\mathcal{G}^{n}}(\omega)$ on the support of $\tilde{P}_{t}^{n}$.  As V is a decreasing function, this implies that, for all $y > 0$,
\begin{align*}
V\left(y\frac{Z_{t \wedge T_{k}}^{\mathcal{G}}}{C}(\omega)\right) \geq V(yZ_{t \wedge T_{k}}^{\mathcal{G}^{n}}(\omega))
\end{align*}
and, therefore,
\begin{align*}
\left[V\left(\frac{yZ_{t \wedge T_{k}}^{\mathcal{G}}(\omega)}{C}\right) > M\right] \supseteq \left[V(yZ_{t \wedge T_{k}}^{\mathcal{G}^{n}}(\omega)) > M\right]
\end{align*}
both restricted to the support of $\tilde{P}_{t}^{n}$.  Thus, for any $M > 0$,
\begin{align*}
E_{P_{n}}[V(yZ_{t \wedge T_{k}}^{\mathcal{G}^{n}})1_{[V(yZ_{t \wedge T_{k}}^{\mathcal{G}^{n}}) > M]}] \leq E_{P_{n}}\left[V\left(y\frac{Z_{t \wedge T_{k}}^{\mathcal{G}}}{C}1_{[V(yZ_{t \wedge T_{k}}^{\mathcal{G}}) > M]}\right) \right]
\end{align*}
Applying the corresponding inequality from Step 3, for $y' = y/C$, we complete this step. \\
\textbf{Step 7} For all $y > 0$
\begin{align*}
\underset{n \rightarrow \infty}{\lim} \left|E_{P_{n}}[V(yZ_{t \wedge T_{k}}^{\mathcal{G}^{n}})] - E_{P_{n}}[V(yZ_{t \wedge T_{k}}^{\mathcal{G}})]\right| = 0
\end{align*}
The argument will change depending on when $V(0) = U(\infty)$ and/or $V(\infty) = U(0)$ are finite-valued.  We present first the argument when $V(0) = U(\infty) = \infty$ and $V(\infty) = U(0) = -\infty$, then sketch how to handle the easier cases where one or the other is finite. \\
Let $\epsilon > 0$ and $y > 0$, and $M > 0$ so our tail inequalities hold.  Then outside of these sets we have $V(yZ_{t \wedge T_{k}}^{\mathcal{G}})(\omega) \in [-M, M]$, and $V(yZ_{t}^{\mathcal{G}^{n}})(\omega) \in [-M, M]$ uniformly in $n$.    Then we see
\begin{align*}
\left|E_{P_{n}}[V(yZ_{t \wedge T_{k}}^{\mathcal{G}^{n}})] - E_{P_{n}}[V(yZ_{t \wedge T_{k}}^{\mathcal{G}})]\right| &\leq E_{P_{n}}\left[\left|V(yZ_{t \wedge T_{k}}^{\mathcal{G}^{n}})\right|1_{[V(yZ_{t \wedge T_{k}}^{\mathcal{G}^{n}}) < -M]}\right] \\
&+ E_{P_{n}}\left[\left|V(yZ_{t \wedge T_{k}}^{\mathcal{G}^{n}})\right|1_{[V(yZ_{t \wedge T_{k}}^{\mathcal{G}^{n}}) > M]}\right] \\
&+ E_{P_{n}}\left[\left|V(yZ_{t \wedge T_{k}}^{\mathcal{G}})\right|1_{[V(yZ_{t \wedge T_{k}}^{\mathcal{G}}) < -M]}\right] \\
&+ E_{P_{n}}\left[\left|V(yZ_{t \wedge T_{k}}^{\mathcal{G}})\right|1_{[V(yZ_{t \wedge T_{k}}^{\mathcal{G}}) > M]}\right] \\
&+ E_{P_{n}}[\left|V^{M}(yZ_{t \wedge T_{k}}^{\mathcal{G}^{n}})] - V^{M}(yZ_{t \wedge T_{k}}^{\mathcal{G}})\right|]
\end{align*}
By dominated convergence and our tail inequalities we see
\begin{align*}
\lim \underset{n \rightarrow \infty}{\sup} \left|E_{P_{n}}[V(yZ_{t \wedge T_{k}}^{\mathcal{G}^{n}})] - E_{P_{n}}[V(yZ_{t \wedge T_{k}}^{\mathcal{G}})]\right| \leq 4 \epsilon
\end{align*}
Letting $\epsilon \downarrow 0$, we achieve the result.  When $V(0)$ and/or $V(\infty)$ are finite, the argument needs a modification.  Suppose $V(0) < \infty$.  This is relevant when $yZ_{t \wedge T_{k}}^{\mathcal{G}}$ and $yZ_{t \wedge T_{k}}^{\mathcal{G}^{n}}$ are both close to zero, which is for paths $\omega$ where $\omega(t)$ is large.  For those paths, $Z_{t \wedge T_{k}}^{\mathcal{G}^{n}}$ can be quite far from $Z_{t \wedge T_{k}}^{\mathcal{G}}$.  However, even if these terms are far apart, $V(yZ_{t \wedge T_{k}}^{\mathcal{G}})$ and $V(yZ_{t \wedge T_{k}}^{\mathcal{G}^{n}})$ will be close together, as each is close to the finite $V(0)$.  A similar argument works for cases where $V(\infty)$ is finite. \\
\textbf{Step 8}  Combine Steps 7 and 3 to obtain $\underset{n \rightarrow \infty}{\lim} u_{n}^{Z_{t \wedge T_{k}}^{\mathcal{G}^{n}}}(x) = u_{k}^{\mathcal{G}}(x)$ for all $x > 0$. \\
Step 3 shows that $v_{n}^{Z_{t \wedge T_{k}}^{\mathcal{G}}}(y) \rightarrow v_{k}^{\mathcal{G}}(y)$ for all $y > 0$ (which is how we concluded that $u_{n}^{Z_{t \wedge T_{k}}^{\mathcal{G}}}(x) \rightarrow u_{k}^{\mathcal{G}}$.  Step 7 then implies that $v_{n}^{Z_{t \wedge T_{k}}^{\mathcal{G}^{n}}}(y) \rightarrow v_{k}^{\mathcal{G}}(y)$ for all $y > 0$. This, in turn, implies $\underset{n \rightarrow \infty}{\lim}u_{n}^{Z_{t \wedge T_{k}}^{\mathcal{G}^{n}}}(x) \rightarrow u_{k}^{\mathcal{G}}(x)$ by standard arguments on conjugate functions.  \\
\textbf{Step 9} Combine Step 8 and Proposition 3.9 to finish the proof in the case where we introduce the stopping times $\{T_{k}\}_{k = 0}^{\infty}$. \\
\textbf{Step 10} Let
\begin{align*}
a_{nk} = E_{P_{n}}[V(yZ_{t \wedge T_{k}}^{\mathcal{G}^{n}})] \\
a_{n} = E_{P_{n}}[V(yZ_{t}^{\mathcal{G}^{n}})] \\
a_{k} = E_{P}[V(yZ_{t \wedge T_{k}}^{\mathcal{G}})] \\
a = E_{P}[V(yZ_{t \wedge T_{k}}^{\mathcal{G}})]
\end{align*}
Note from our previous results and definitions that
\begin{align*}
\underset{n \rightarrow \infty}{\lim} a_{nk} = a_{k} \text{ for each k by Steps 1-7} \\
\underset{k \rightarrow \infty}{\lim} a_{nk} = a_{n} \text{ for each n by definition of the stopping times} \\
\underset{k \rightarrow \infty}{\lim} a_{k} = a \text{ by Dominated Convergence Theorem}
\end{align*}
As convergence is uniform in one variable by assumption 5., we may interchange the following double limit, giving us (for example)
\begin{align*}
\underset{n \rightarrow \infty}{\lim}E_{P_{n}}[V(yZ_{t}^{\mathcal{G}^{n}})] &= \underset{n \rightarrow \infty}{\lim} a_{n} \\
&= \underset{n \rightarrow \infty}{\lim} \underset{k \rightarrow \infty}{\lim}a_{nk} \\
&= \underset{k \rightarrow \infty}{\lim} \underset{n \rightarrow \infty}{\lim}a_{nk} \\
&= \underset{k \rightarrow \infty}{\lim} a_{k} \\
&= a \\
&= E_{P}[V(yZ_{t}^{\mathcal{G}})]
\end{align*} \\
\textbf{Step 11}  From our last step we see $v_{n}^{Z_{t}^{\mathcal{G}^{n}}}(y) \rightarrow v^{\mathcal{G}}(y)$ as $n \rightarrow \infty$ for all $y > 0$.  Using standard arguments of conjugate functions we see $u_{n}^{Z_{t}^{\mathcal{G}^{n}}}(x) \rightarrow u^{\mathcal{G}}(x)$ as $n \rightarrow \infty$ for all $x > 0$.  Then for each $x > 0$
\begin{align*}
\lim \underset{n \rightarrow \infty}{\sup}u_{n}^{\mathcal{G}^{n}}(x) \leq \lim \underset{n \rightarrow \infty}{\sup} u_{n}^{Z_{t}^{\mathcal{G}^{n}}}(x) = u^{\mathcal{G}}(x)
\end{align*}
Proposition 3.9 gives us the other direction so the proof is complete.
\end{proof}
\begin{example}
Recall that the martingale preserving measure $\tilde{P}_{t}$ exists on a time interval $[0,t]$, where $t$ is fixed.  The simplest example is to take $Y = S_{1}/S_{t} \overset{d}{=} \exp{(\omega(1) - \omega(t))} \overset{d}{=} \exp{\omega(1-t)}$, and $Y_{n} = \exp{X_{1}^{n}}/\exp{X_{t}^{n}} \overset{d}{\approx} \exp{\frac{\sqrt{1-t}}{\sqrt{\lfloor n(1-t) \rfloor}} S_{\lfloor n(1-t) \rfloor}}$ as $n \rightarrow \infty$.  Then for $s < t$.
\begin{align*}
\eta_{s}^{y} &= \frac{P[e^{\omega(1)}/e^{\omega(t)} \in dy | \mathcal{F}_{s}]}{P[e^{\omega(1)} \in dy]} \\
&=\frac{ P[e^{\omega(1) - \omega(t)} \in dy | \mathcal{F}_{s}]}{P[e^{\omega(1)} \in dy]} \\
&\overset{d}{=}\frac{P[e^{\omega(1-t)} \in dy | \mathcal{F}_{s}]}{P[e^{\omega(1)} \in dy]} \\
&= \frac{P[e^{\omega(1-t)} \in dy]}{P[e^{\omega(1)} \in dy]}
\end{align*}
Similarly,
\begin{align*}
\eta_{t}^{y,n} &= \frac{P[e^{X_{1}^{(n)} - X_{t}^{(n)}} \in dy | \mathcal{F}_{s}]}{P[e^{X_{1}^{(n)}} \in dy]} \\
&\overset{d}{=} \frac{P[e^{X_{1-t}^{(n)}} \in dy | \mathcal{F}_{s}]}{P[e^{X_{1}^{n}} \in dy]} \\
&= \frac{P[e^{X_{1-t}^{(n)}} \in dy]}{P[e^{X_{1}^{n}} \in dy]}
\end{align*}
Using Theorem 3.21 which gives us the distribution of a simple random walk, as well as the distribution for a normal random variable, we find
\begin{align*}
\frac{\eta_{s}^{y,n}}{\eta_{s}^{y}}\approx \frac{\binom{\lfloor n(1-t) \rfloor}{\frac{\lfloor n(1-t) \rfloor + y}{2}}2^{-\lfloor n(1-t) \rfloor}}{\frac{1}{\sqrt{2 \pi (1-t)}}e^{\frac{-y^{2}}{2(1-t)}}} \cdot \frac{\frac{1}{\sqrt{2 \pi}}e^{\frac{-y^{2}}{2}}}{\binom{n}{\frac{n + y}{2}}2^{-n}}
\end{align*}
In this example we see that as $n \rightarrow \infty$ this fraction stays uniformly bounded in $y, \omega$, so assumptions 3, and 4 hold for $n$ sufficiently large.
\end{example}
\end{subsection}
\end{section}
\begin{section}{Conclusion}
In this paper we studied an informed insider that trades in an arbitrage free financial market, maximizing expected utility of wealth on a given time horizon.  We assumed this informed insider possessed extra information of some function $Y$ of the future prices of a stock. In particular, we studied the "strong approach" where the investor knows the functional $\omega$ by $\omega$.  This modeling approach using the initial enlargement of a filtration is developed in \cite{MR604176}. \\
We then used this framework to extend Kreps' conjecture \cite{MR4154768} that optimal expected utility for a sequence of discrete-time economies in their discrete-time enlarged filtrations will approach the optimal utility of the Black-Scholes-Merton economy in its enlarged filtration under certain conditions. \\
In the future, we will aim to extend this result using the "weak approach," when the investor knows the law of the functional $Y$ under the effective probability of the market assumed to be unknown.  This notion of weak information is defined in \cite{MR2021790}, and further developed in \cite{MR3954302, MR2213259}.
\end{section}
\bibliography{Lindsell_2022_bib}
\bibliographystyle{plain}
\nocite{*}
\end{document}